\documentclass[12pt]{amsart}
\usepackage{amsmath,amsthm,amsfonts,amssymb,mathrsfs}
\date{\today}

\input xymatrix
\xyoption{all}

\usepackage{color}

\usepackage{hyperref}

\newtheorem{theorem}{Theorem}[section]

\newtheorem{proposition}[theorem]{Proposition}
\newtheorem{corollary}[theorem]{Corollary}
\newtheorem{lemma}[theorem]{Lemma}
\theoremstyle{definition}
\newtheorem{remark}[theorem]{Remark}

\begin{document}

\title[The monoid of order isomorphisms of principal filters ...]{The monoid of order isomorphisms of principal filters of a power of the positive integers}

\author[Oleg~Gutik and Taras~Mokrytskyi]{Oleg~Gutik and Taras~Mokrytskyi}
\address{Faculty of Mathematics, National University of Lviv,
Universytetska 1, Lviv, 79000, Ukraine}
\email{ogutik@gmail.com, ovgutik@yahoo.com, tmokrytskyi@gmail.com}



\begin{abstract}
Let $n$ be any positive integer and $\mathscr{I\!\!P\!F}(\mathbb{N}^n)$ be the semigroup of all order isomorphisms between principal filters of the $n$-th power of the set of positive integers $\mathbb{N}$ with the product order. We study algebraic properties of the semigroup $\mathscr{I\!\!P\!F}(\mathbb{N}^n)$. In particular, we show that $\mathscr{I\!\!P\!F}(\mathbb{N}^n)$ is a bisimple, $E$-unitary, $F$-inverse semigroup, describe Green's relations on $\mathscr{I\!\!P\!F}(\mathbb{N}^n)$ and its maximal subgroups. We show that
the semigroup $\mathscr{I\!\!P\!F}(\mathbb{N}^n)$ is isomorphic to the semidirect product of the direct $n$-th power of the bicyclic monoid ${\mathscr{C}}^n(p,q)$ by the group of permutation  $\mathscr{S}_n$. Also we prove that
every non-identity congruence $\mathfrak{C}$ on the semigroup $\mathscr{I\!\!P\!F}(\mathbb{N}^n)$ is group and describe the least group congruence on $\mathscr{I\!\!P\!F}(\mathbb{N}^n)$. We show that every Hausdorff shift-continuous topology on $\mathscr{I\!\!P\!F}(\mathbb{N}^n)$ is discrete and discuss embedding of the semigroup $\mathscr{I\!\!P\!F}(\mathbb{N}^n)$ into compact-like topological semigroups.

\end{abstract}

\maketitle

\section{Introduction and preliminaries}

We shall follow the terminology of~\cite{Carruth-Hildebrant-Koch-1983-1986, Clifford-Preston-1961-1967, Engelking-1989, Petrich-1984, Ruppert-1984}.

In this paper we shall denote the first infinite ordinal by $\omega$, the set of integers by $\mathbb{Z}$, the set of positive integers by $\mathbb{N}$, the set of non-negative integers by $\mathbb{N}_0$, the additive group of integers by $\mathbb{Z}(+)$ and the symmetric group of degree $n$ by $\mathscr{S}_n$, i.e., $\mathscr{S}_n$ is the group of all permutations of an $n$-element set. All topological spaces, considered in this paper, are assumed to be  Hausdorff.

Let $(X,\leqslant)$ be a partially ordered set (a poset). For an arbitrary $x\in X$ we denote
 \begin{equation*}
{\uparrow}x=\left\{y\in X\colon x\leqslant y\right\} \qquad \hbox{and} \qquad {\downarrow}x=\left\{y\in X\colon y\leqslant x\right\}.
\end{equation*}
The sets ${\uparrow}x$ and ${\downarrow}x$ are called the \emph{principal filter} and the \emph{principal ideal}, respectively, generated by the element $x\in X$. A map $\alpha\colon (X,\leqslant)\to(Y,\eqslantless)$ from poset $(X,\leqslant)$ into a poset $(Y,\eqslantless)$ is called \emph{monotone} or \emph{order preserving} if $x\leqslant y$ in $(X,\leqslant)$ implies that $x\alpha\eqslantless y\alpha$ in $(Y,\eqslantless)$. A monotone map $\alpha\colon (X,\leqslant)\to(Y,\eqslantless)$ is said to be an \emph{order isomorphism} if it is bijective and its converse $\alpha^{-1}\colon(Y,\eqslantless)\to(X,\leqslant)$ is monotone.

An semigroup $S$ is called {\it inverse} if for any element $x\in S$ there exists a unique $x^{-1}\in S$ such that $xx^{-1}x=x$ and $x^{-1}xx^{-1}=x^{-1}$. The element $x^{-1}$ is called the {\it inverse of} $x\in S$. If $S$ is an inverse semigroup, then the function $\operatorname{inv}\colon S\to S$ which assigns to every element $x$ of $S$ its inverse element $x^{-1}$ is called the {\it inversion}.

A congruence $\mathfrak{C}$ on a semigroup $S$ is called \emph{non-trivial} if $\mathfrak{C}$ is distinct from universal and identity congruences on $S$, and a \emph{group congruence} if the quotient semigroup $S/\mathfrak{C}$ is a group.

If $S$ is a semigroup, then we shall denote the subset of all idempotents in $S$ by $E(S)$. If $S$ is an inverse semigroup, then $E(S)$ is closed under multiplication and we shall refer to $E(S)$ as a \emph{band} (or the \emph{band of} $S$). Then the semigroup operation on $S$ determines the following partial order $\preccurlyeq$
on $E(S)$: $e\preccurlyeq f$ if and only if $ef=fe=e$. This order is called the {\em natural partial order} on $E(S)$. A \emph{semilattice} is a commutative semigroup of idempotents. A semilattice $E$ is called {\em linearly ordered} or a \emph{chain}
if its natural  order is a linear order.   By
$(\mathscr{P}(X),\cap)$ we shall denote the
semilattice of all subsets of a set
$X$ with the semilattice
operation of ``intersection''.

If $S$ is a semigroup, then we shall denote the Green relations on $S$ by $\mathscr{R}$, $\mathscr{L}$, $\mathscr{J}$, $\mathscr{D}$ and $\mathscr{H}$ (see \cite{Clifford-Preston-1961-1967}). A semigroup $S$ is called \emph{simple} if $S$ does not contain proper two-sided ideals and \emph{bisimple} if $S$ has only one $\mathscr{D}$-class.

Hereafter we shall assume that $\lambda$ is an infinite cardinal.
If $\alpha\colon \lambda\rightharpoonup \lambda$ is a partial map, then we shall denote
the domain and the range of $\alpha$ by $\operatorname{dom}\alpha$ and $\operatorname{ran}\alpha$, respectively.

Let $\mathscr{I}_\lambda$ denote the set of all partial one-to-one
transformations of an infinite set $X$ of cardinality $\lambda$
together with the following semigroup operation:
\begin{equation*}
x(\alpha\beta)=(x\alpha)\beta \quad \hbox{if} \quad
x\in\operatorname{dom}(\alpha\beta)=\left\{
y\in\operatorname{dom}\alpha\mid
y\alpha\in\operatorname{dom}\beta\right\},\quad \hbox{for }
\alpha,\beta\in\mathscr{I}_\lambda.
\end{equation*}
The semigroup $\mathscr{I}_\lambda$ is called the \emph{symmetric
inverse semigroup} over the set $\lambda$~(see \cite[Section~1.9]{Clifford-Preston-1961-1967}).
The symmetric inverse semigroup was introduced by
Wagner~\cite{Vagner-1952} and it plays a major role in the theory of
semigroups.

The \emph{bicyclic semigroup} (or the \emph{bicyclic monoid}) ${\mathscr{C}}(p,q)$ is the semigroup with the identity $1$ generated by elements $p$ and $q$ and the relation $pq=1$.

\begin{remark}\label{remark-1.1}
We observe that the bicyclic semigroup is isomorphic to the
semigroup $\mathscr{C}_{\mathbb{N}}(\alpha,\beta)$ which is
generated by partial transformations $\alpha$ and $\beta$ of the set
of positive integers $\mathbb{N}$, defined as follows:
$(n)\alpha=n+1$ if $n\geqslant 1$ and $(n)\beta=n-1$ if $n> 1$ (see Exercise~IV.1.11$(ii)$ in \cite{Petrich-1984}).
\end{remark}

If $T$ is a semigroup, then we say that a subsemigroup $S$ of $T$ is a \emph{bicyclic subsemigroup of} $T$ if $S$ is isomorphic to the bicyclic semigroup ${\mathscr{C}}(p,q)$.

\smallskip

A (\emph{semi})\emph{topological} \emph{semigroup} is a topological space with a (separately) continuous semigroup operation. An inverse topological semigroup with continuous inversion is called a \emph{topological inverse semigroup}.

A topology $\tau$ on a semigroup $S$ is called:
\begin{itemize}
  \item \emph{semigroup} if $(S,\tau)$ is a topological semigroup;
  \item \emph{shift-continuous} if $(S,\tau)$ is a semitopological semigroup.
\end{itemize}

The bicyclic semigroup is bisimple and every one of its congruences is either trivial or a group congruence. Hence, every homomorphism of the bicyclic semigroup is either an isomorphism or a group homomorphism, and moreover every non-isomorphic image of ${\mathscr{C}}(p,q)$  is a cyclic group~(see \cite[Corollary~1.32]{Clifford-Preston-1961-1967}). The bicyclic semigroup plays an important role in algebraic theory of semigroups and in the theory of topological semigroups. For example a well-known Andersen's result~\cite{Andersen-1952} states that a ($0$--)simple semigroup with an idempotent is completely ($0$--)simple if and only if it does not contain an isomorphic copy of the bicyclic semigroup. The bicyclic monoid admits only the discrete semigroup Hausdorff topology. Bertman and  West in \cite{Bertman-West-1976} extended this result for the case of Hausdorff semitopological semigroups. Stable and $\Gamma$-compact topological semigroups do not contain the bicyclic monoid~\cite{Anderson-Hunter-Koch-1965, Hildebrant-Koch-1986}. The problem of embedding the bicyclic monoid into compact-like topological semigroups was studied in \cite{Banakh-Dimitrova-Gutik-2009, Banakh-Dimitrova-Gutik-2010, Gutik-Repovs-2007}.
Independently to Eberhart-Selden results on topolozabilty of the bicyclic semigroup, in \cite{Taimanov-1973} Taimanov constructed a commutative semigroup $\mathfrak{A}_\kappa$ of cardinality $\kappa$ which admits only the discrete semigroup topology. Also, Taimanov \cite{Taimanov-1975} gave sufficient conditions for a commutative semigroup to have a non-discrete semigroup topology. In the paper \cite{Gutik-2016} it was showed that for the Taimanov semigroup $\mathfrak{A}_\kappa$ from \cite{Taimanov-1973} the following conditions hold:
every semigroup $T_1$-topology $\tau$ on the semigroup $\mathfrak{A}_\kappa$ is discrete; $\mathfrak A_\kappa$ is closed in any $T_1$-topological semigroup containing $\mathfrak A_\kappa$
and every homomorphic non-isomorphic image of $\mathfrak{A}_\kappa$ is a zero-semigroup.

Let $\alpha$ be an arbitrary ordinal, let $H_\alpha$ denote the set of all ordinal numbers less than $\omega^\alpha$, let $+$ denote usual
ordinal addition, and let $W_\alpha = H_\alpha\times H_\alpha$. An operation is defined on $W_\alpha$ by
\begin{equation*}
  (\beta,\gamma)(\delta,\eta)=(\beta + (\max\{\gamma,\delta\} -\gamma),\eta + (\max\{\gamma,\delta\} - \delta)).
\end{equation*}
Then  under this operation, $W_\alpha$ is a bisimple inverse semigroup with identity $(0,0)$ which is called  the
\emph{$\alpha$-bicyclic semigroup} \cite{Hogan-1973}. Note that $W_1$ is isomorphic to the bicyclic semigroup.
In \cite{Hogan-1984} Hogan proved that for any ordinal $\alpha$ every semigroup Hausdorff locally compact topology on the $\alpha$-bicyclic semigroup is discrete. Also, in \cite{Selden_A-1985} Annie Selden constructed a non-discrete semigroup Hausdorff topology on the $2$-bicyclic semigroup. But Bardyla found a gap in the proof of Theorem 2.9 of \cite{Hogan-1984} and in \cite{Bardyla-2016a} he constructed a non-discrete semigroup Hausdorff locally compact topology on the $\omega+1$-bicyclic semigroup. Also, in \cite{Bardyla-2016a} Bardyla showed that the statement of Theorem 2.9 from \cite{Hogan-1984} is true in the case when $\alpha\leqslant\omega$ and in \cite{Bardyla-2018} for every ordinal $\alpha\leqslant\omega$  he described all shift-continuous locally compact Hausdorff topologies on the $\alpha$-bicyclic monoid.

In \cite{Fihel-Gutik-2011} it is proved that the discrete topology is the unique shift-continuous Hausdorff topology on the extended bicyclic semigroup $\mathscr{C}_{\mathbb{Z}}$. We observe that for many ($0$-)bisimple semigroups   $S$ the following statement holds: \emph{every shift-continuous Hausdorff Baire (in particular locally compact) topology on $S$ is discrete} (see \cite{Chuchman-Gutik-2010, Gutik-2018, Gutik-Maksymyk-2016a, Gutik-Pozdnyakova-2014, Gutik-Repovs-2011, Gutik-Repovs-2012}).

A graph inverse semigroup $G(E)$ is a semigroup constructed from a directed graph $E$, where, roughly speaking, elements correspond to paths in the graph. These semigroups were introduced by Ash and Hall in \cite{Ash-Hall-1975} in order to show that every partial order can be realized as that of the nonzero $\mathscr{J}$-classes of an inverse semigroup. Graph inverse semigroups also generalize polycyclic monoids, first defined by Nivat and Perrot in \cite{Nivat-Perrot-1970}.
In the paper \cite{Mesyan-Mitchell-Morayne-Peresse-2016} Mesyan, Mitchell, Morayne and P\'{e}resse showed that if $E$ is a finite graph, then the only locally compact Hausdorff semigroup topology on the graph inverse semigroup $G(E)$ is the discrete topology. In \cite{Bardyla-Gutik-2016} it was proved that the conclusion of this statement also holds for graphs $E$ consisting of one vertex and infinitely many loops (i.e., infinitely generated polycyclic monoids). We observe that graph inverse semigroups which admit
only discrete locally compact semigroup topology were characterized in \cite{Bardyla-20??}. An amazing dichotomy for the bicyclic monoid with adjoined zero $\mathscr{C}^0={\mathscr{C}}(p,q)\sqcup\{0\}$ was proved in \cite{Gutik-2015}: \emph{every Hausdorff locally compact semitopological bicyclic monoid $\mathscr{C}^0$ with adjoined zero is either compact or discrete}.
The above dichotomy was extended by Bardyla in \cite{Bardyla-2016} to locally compact $\lambda$-polycyclic semitopological monoids, in \cite{Bardyla-2018a} to locally compact semitopological graph inverse semigroups and to locally compact semitopological interassociates of the bicyclic monoid with adjoined zero \cite{Gutik-Maksymyk-2016}.

We observe that some classes of inverse semigroups of partial transformations with cofinite domains and ranges have algebraic properties similar to the bicyclic semigroup.
In  \cite{Gutik-Repovs-2015} it was shown that the semigroup
$\mathscr{I}^{\mathrm{cf}}_\lambda$ of injective partial cofinite selfmaps of a cardinal $\lambda$  is a bisimple inverse semigroup
and that for every non-empty chain $L$ in
$E(\mathscr{I}^{\mathrm{cf}}_\lambda)$ there exists an inverse
subsemigroup $S$ of $\mathscr{I}^{\mathrm{cf}}_\lambda$ such that
$S$ is isomorphic to the bicyclic semigroup and $L\subseteq E(S)$,
and proved that every non-trivial congruence on
$\mathscr{I}^{\mathrm{cf}}_\lambda$ is a group congruence. Also, the structure of the quotient semigroup $\mathscr{I}^{\mathrm{cf}}_\lambda/\mathfrak{C}_{\mathbf{mg}}$, where $\mathfrak{C}_{\mathbf{mg}}$ is the least group congruence on $\mathscr{I}^{\mathrm{cf}}_\lambda$, was described there.

The semigroups $\mathscr{I}_{\infty}^{\!\nearrow}(\mathbb{N})$ and $\mathscr{I}_{\infty}^{\!\nearrow}(\mathbb{Z})$ of injective monotone partial selfmaps with cofinite domains and images of positive integers and integers, respectively, were studied in \cite{Gutik-Repovs-2011} and \cite{Gutik-Repovs-2012}. There it was proved that the semigroups $\mathscr{I}_{\infty}^{\!\nearrow}(\mathbb{N})$ and $\mathscr{I}_{\infty}^{\!\nearrow}(\mathbb{Z})$ are bisimple and  every non-trivial homomorphic image of $\mathscr{I}_{\infty}^{\!\nearrow}(\mathbb{N})$ and $\mathscr{I}_{\infty}^{\!\nearrow}(\mathbb{Z})$ are isomorphic to  $\mathbb{Z}(+)$  and $\mathbb{Z}(+)\times\mathbb{Z}(+)$, respectively.

In \cite{Gutik-Pozdnyakova-2014} Gutik and Podniakova studied  the semigroup
$\mathscr{I\!O}\!_{\infty}(\mathbb{Z}^n_{\operatorname{lex}})$ of monotone injective partial selfmaps of the set of $L_n\times_{\operatorname{lex}}\mathbb{Z}$ having cofinite domain and image, where $L_n\times_{\operatorname{lex}}\mathbb{Z}$ is the lexicographic product of $n$-element chain and the set of integers with the usual linear order. There it was shown that the semigroup $\mathscr{I\!O}\!_{\infty}(\mathbb{Z}^n_{\operatorname{lex}})$ is
bisimple, $\mathscr{I\!O}\!_{\infty}(\mathbb{Z}^n_{\operatorname{lex}})$ is finitely generated, every automorphism of $\mathscr{I\!O}\!_{\infty}(\mathbb{Z})$ is inner and showed that in the case $n\geqslant 2$ the semigroup $\mathscr{I\!O}\!_{\infty}(\mathbb{Z}^n_{\operatorname{lex}})$ has non-inner automorphisms. Also, in \cite{Gutik-Pozdnyakova-2014} projective congruences of the semigroup $\mathscr{I\!O}\!_{\infty}(\mathbb{Z}^n_{\operatorname{lex}})$ were studied.
In \cite{Gutik-Pozdnyakova-2014} it was proved that for every positive integer $n$ the quotient semigroup
$\mathscr{I\!O}\!_{\infty}(\mathbb{Z}^n_{\operatorname{lex}})/\mathfrak{C}_{\mathbf{mg}}$, where $\mathfrak{C}_{\mathbf{mg}}$ is a least group congruence on $\mathscr{I\!O}\!_{\infty}(\mathbb{Z}^n_{\operatorname{lex}})$, is isomorphic to the direct power $\left(\mathbb{Z}(+)\right)^{2n}$. The structure of the sublattice of congruences on $\mathscr{I\!O}\!_{\infty}(\mathbb{Z}^n_{\operatorname{lex}})$ which are contained in the least group congruence is described in \cite{Gutik-Pozdniakova-2014}.

On the other hand the semigroup $\mathscr{P\!O}\!_{\infty}(\mathbb{N}^2_{\leqslant})$ of monotone injective partial selfmaps with cofinite domains and images of the square $\mathbb{N}^2$ with the product order has more complicated structure \cite{Gutik-Pozdniakova-2016}.
In particular, there was proved that $\mathscr{D}=\mathscr{J}$  in $\mathscr{P\!O}\!_{\infty}(\mathbb{N}^2_{\leqslant})$. In \cite{Gutik-Pozdniakova-2016a} it was proved that the natural partial order $\preccurlyeq$ on $\mathscr{P\!O}\!_{\infty}(\mathbb{N}^2_{\leqslant})$ coincides with the natural partial order which is induced from symmetric inverse monoid $\mathscr{I}_{\mathbb{N}\times\mathbb{N}}$ over  $\mathbb{N}\times\mathbb{N}$ onto the semigroup $\mathscr{P\!O}\!_{\infty}(\mathbb{N}^2_{\leqslant})$. The congruence $\sigma$ on the semigroup $\mathscr{P\!O}\!_{\infty}(\mathbb{N}^2_{\leqslant})$, which is generated by the natural order $\preccurlyeq$ on the semigroup $\mathscr{P\!O}\!_{\infty}(\mathbb{N}^2_{\leqslant})$ is described: $\alpha\sigma\beta$ if and only if $\alpha$ and $\beta$ are comparable in $\left(\mathscr{P\!O}\!_{\infty}(\mathbb{N}^2_{\leqslant}),\preccurlyeq\right)$. Also, there was proved that the quotient semigroup $\mathscr{P\!O}\!_{\infty}(\mathbb{N}^2_{\leqslant})/\sigma$ is isomorphic to the semidirect product of the free commutative monoid $\mathfrak{AM}_\omega$ over an infinite countable set by the cyclic group of the order two $\mathbb{Z}_2$.

For an arbitrary positive integer $n$ by $\left(\mathbb{N}^n,\leqslant\right)$ we denote the $n$-th power of the set of positive integers $\mathbb{N}$ with the product order:
\begin{equation*}
  \left(x_1,\ldots,x_n\right)\leqslant\left(y_1,\ldots,y_n\right) \qquad \hbox{if and only if} \qquad x_i\leq y_i \quad \hbox{for all} \quad i=1,\ldots,n.
\end{equation*}
It is obvious that the set of all order isomorphisms between principal filters of the poset $\left(\mathbb{N}^n,\leqslant\right)$ with the operation of composition of partial maps form a semigroup. This semigroup will be denoted by $\mathscr{I\!\!P\!F}(\mathbb{N}^n)$. It is easy to see that the semigroup $\mathscr{I\!\!P\!F}(\mathbb{N}^n)$ is isomorphic to the Munn semigroup of $\mathbb{N}^n$ with the dual order to $\leqslant$ (see \cite[Section~5.4]{Howie-1995} or \cite{Munn=1966}). Also, Remark~\ref{remark-1.1} implies that the semigroup $\mathscr{I\!\!P\!F}(\mathbb{N}^n)$ is some generalization of the bicyclic semigroup ${\mathscr{C}}(p,q)$. Hence it is natural to ask: \emph{what algebraic and topological properties of the semigroup $\mathscr{I\!\!P\!F}(\mathbb{N}^n)$ are similar to ones of the bicyclic monoid?}

Later by $\mathbb{I}$ we shall denote the identity map of $\mathbb{N}^n$. It is obvious that $\mathbb{I}$ is a unit element of the semigroup $\mathscr{I\!\!P\!F}(\mathbb{N}^n)$. By $H(\mathbb{I})$ we shall denote the group of units of $\mathscr{I\!\!P\!F}(\mathbb{N}^n)$. It is clear that $\alpha\in \mathscr{I\!\!P\!F}(\mathbb{N}^n)$ is an element of $H(\mathbb{I})$ if and only if it is an order isomorphism of the poset $\left(\mathbb{N}^n,\leqslant\right)$.

In this paper we study algebraic properties of the semigroup $\mathscr{I\!\!P\!F}(\mathbb{N}^n)$. In particular, we show that $\mathscr{I\!\!P\!F}(\mathbb{N}^n)$ is a bisimple, $E$-unitary, $F$-inverse semigroup and describe Green's relations on $\mathscr{I\!\!P\!F}(\mathbb{N}^n)$. We prove that
the semigroup $\mathscr{I\!\!P\!F}(\mathbb{N}^n)$ is isomorphic to the semidirect product of the $n$-th direct power of the bicyclic monoid ${\mathscr{C}}(p,q)^n$ by the group of permutation  $\mathscr{S}_n$. We prove that
every non-identity congruence $\mathfrak{C}$ on the semigroup $\mathscr{I\!\!P\!F}(\mathbb{N}^n)$ is a group congruence  and describe the least group congruence on $\mathscr{I\!\!P\!F}(\mathbb{N}^n)$. Also, we show that every Hausdorff shift-continuous topology on $\mathscr{I\!\!P\!F}(\mathbb{N}^n)$ is discrete and discuss embedding of the semigroup $\mathscr{I\!\!P\!F}(\mathbb{N}^n)$ into compact-like topological semigroups.

\section{Algebraic properties of the semigroup $\mathscr{I\!\!P\!F}(\mathbb{N}^n)$}

By $\mathscr{P}_{\uparrow}(\mathbb{N}^n)=\left\{{\uparrow}\mathbf{x}\colon \mathbf{x}\in \mathbb{N}^n\right\}$ we denote the set of all principal filters of the poset $\left(\mathbb{N}^n,\leqslant\right)$. It is obvious that $(\mathscr{P}_{\uparrow}(\mathbb{N}^n),\cap)$ is a subsemilattice of $(\mathscr{P}(\mathbb{N}^n),\cap)$. Also, we observe that the semilattice $\left(\mathbb{N}^n,\max\right)$, which is the set $\mathbb{N}^n$ with the point-wise operation $\max$:
\begin{equation*}
  \left(x_1,\ldots,x_n\right)\left(y_1,\ldots,y_n\right)=\left(\max\left\{x_1,y_1\right\},\ldots,\max\left\{x_n,y_n\right\}\right),
\end{equation*}
is isomorphic to the semilattice $(\mathscr{P}_{\uparrow}(\mathbb{N}^n),\cap)$ by the mapping $\left(x_1,\ldots,x_n\right)\mapsto{\uparrow}\left(x_1,\ldots,x_n\right)$.

\begin{proposition}\label{proposition-2.1}
Let $n$ be any positive integer. Then the following statements hold:
\begin{itemize}
  \item[$(i)$] $\mathscr{I\!\!P\!F}(\mathbb{N}^n)$ is an inverse semigroup;
  \item[$(ii)$] the semilattice $E(\mathscr{I\!\!P\!F}(\mathbb{N}^n))$ is isomorphic to $(\mathscr{P}_{\uparrow}(\mathbb{N}^n),\cap)$ by the mapping $\varepsilon\mapsto\operatorname{dom}\varepsilon$, and hence it is isomorphic to the semilattice $\left(\mathbb{N}^n,\max\right)$;
  \item[$(iii)$] $\alpha\mathscr{L}\beta$ in $\mathscr{I\!\!P\!F}(\mathbb{N}^n)$ if and only if $\operatorname{dom}\alpha=\operatorname{dom}\beta$;
  \item[$(iv)$] $\alpha\mathscr{R}\beta$ in $\mathscr{I\!\!P\!F}(\mathbb{N}^n)$ if and only if $\operatorname{ran}\alpha=\operatorname{ran}\beta$;
  \item[$(v)$] $\alpha\mathscr{H}\beta$ in $\mathscr{I\!\!P\!F}(\mathbb{N}^n)$ if and only if $\operatorname{dom}\alpha=\operatorname{dom}\beta$ and $\operatorname{ran}\alpha=\operatorname{ran}\beta$;
  \item[$(vi)$] for any idempotents $\varepsilon,\iota\in \mathscr{I\!\!P\!F}(\mathbb{N}^n)$ there exist elements $\alpha,\beta\in
         \mathscr{I\!\!P\!F}(\mathbb{N}^n)$ such that $\alpha\beta=\varepsilon$ and $\beta\alpha=\iota$;
  \item[$(vii)$] $\mathscr{I\!\!P\!F}(\mathbb{N}^n)$ is bisimple and hence it is simple.
\end{itemize}
\end{proposition}

\begin{proof}
It it obvious that $\mathscr{I\!\!P\!F}(\mathbb{N}^n)$ is an inverse subsemigroup of the symmetric inverse monoid $\mathscr{I}_{\mathbb{N}^n}$ over the set $\mathbb{N}^n$. Then statements $(i)$--$(v)$ follow from the definitions of the semigroup $\mathscr{I\!\!P\!F}(\mathbb{N}^n)$ and Green's relations.

$(vi)$ Fix arbitrary idempotents $\varepsilon,\iota\in\mathscr{I\!\!P\!F}(\mathbb{N}^n)$. Since $\operatorname{dom}\varepsilon$ and $\operatorname{dom}\iota$ are principal filters in $\left(\mathbb{N}^n,\leqslant\right)$ we put $\operatorname{dom}\varepsilon$ and $\operatorname{dom}\iota$ are generated by elements $\left(x_1^0,\ldots,x_n^0\right)$ and $\left(y_1^0,\ldots,y_n^0\right)$ of the poset $\left(\mathbb{N}^n,\leqslant\right)$. We define a partial map $\alpha\colon \mathbb{N}^n\rightharpoonup \mathbb{N}^n$ in the following way:
\begin{equation*}
  \operatorname{dom}\alpha=\operatorname{dom}\varepsilon, \qquad \operatorname{ran}\alpha=\operatorname{dom}\iota \qquad \hbox{and}
\end{equation*}
\begin{equation*}
  \left(z_1,\ldots,z_n\right)\alpha=\left(z_1-x_1^0+y_1^0,\ldots,z_n-x_n^0+y_n^0\right), \quad \hbox{for any} \quad \left(z_1,\ldots,z_n\right)\in\operatorname{dom}\alpha.
\end{equation*}
Then $\alpha\alpha^{-1}=\varepsilon$ and $\alpha^{-1}\alpha=\iota$ and hence we put $\beta=\alpha^{-1}$.

Statement $(vii)$ follows from $(vi)$ and Proposition~3.2.5$(1)$ of \cite{Lawson-1998}.
\end{proof}

The proofs of the following two lemmas are trivial.

\begin{lemma}\label{lemma-2.2}
Let $n$ be a positive integer. Then for any  $i=1,\ldots,n$ the projection on $i$-th coordinate
\begin{equation*}
\pi_i\colon\mathbb{N}^n\to \mathbb{N}^n\colon (x_1,\ldots,\underbrace{{x_i}}_{i\hbox{\footnotesize{-th}}},\ldots,x_n)\longmapsto(1,\ldots,\underbrace{x_i}_{i\hbox{\footnotesize{-th}}},\ldots,1)
\end{equation*}
 is a monotone map and moreover $(1,\ldots,\underbrace{x_i}_{i\hbox{\footnotesize{-th}}},\ldots,1)\leqslant (x_1,\ldots,\underbrace{{x_i}}_{i\hbox{\footnotesize{-th}}},\ldots,x_n)$ in $\left(\mathbb{N}^n,\leqslant\right)$.
\end{lemma}

\begin{lemma}\label{lemma-2.3}
Every permutation $\sigma$ of the set $\{1,\dots,n\}$ induces an order isomorphism $\alpha^*\colon\mathbb N^n\to\mathbb N^n$, $\alpha^*\colon x\mapsto x\circ\alpha$, of the poset $\left(\mathbb{N}^n,\leqslant\right)$.
\end{lemma}

\begin{lemma}\label{lemma-2.4}
Let $n$ be a positive integer and a map $\alpha\colon \mathbb{N}^n\to \mathbb{N}^n$ be an order isomorphism of the poset $\left(\mathbb{N}^n,\leqslant\right)$ such that $\left(x_i^2\right)\alpha=x_i^2$ for any $i=1,\ldots,n$, where $x_i^2=(1,\ldots,\underbrace{{2}}_{i\hbox{\footnotesize{-th}}},\ldots,1)$ is the element of the poset $\left(\mathbb{N}^n,\leqslant\right)$ such that only $i$-th coordinate of $x_i^2$ is equal to $2$ and all other coordinates are equal to $1$. Then $\alpha$ is the identity map of $\left(\mathbb{N}^n,\leqslant\right)$.
\end{lemma}

\begin{proof}
We observe that the element $(1,1,\ldots,1)$ is the minimum in the poset $\left(\mathbb{N}^n,\leqslant\right)$ and since $\alpha$ is an order isomorphism of $\left(\mathbb{N}^n,\leqslant\right)$ we have that $(1,1,\ldots,1)\alpha=(1,1,\ldots,1)$.

Now, by induction we get that for any positive integer $k\geqslant2$ for the element $x_1^k=(k,1,1,\ldots,1)$ of $\left(\mathbb{N}^n,\leqslant\right)$ the set ${\downarrow}x_1^k$ is a $k$-element chain, then so is $({\downarrow}x_1^k)\alpha$. This implies that only first coordinate of $(x_1^k)\alpha$ is equal to $k$ and all other coordinates are equal to $1$, because in other cases we have that ${\downarrow}((x_1^k)\alpha)$ is not a $k$-element chain in $\left(\mathbb{N}^n,\leqslant\right)$. Similarly, for every positive $i=1,\ldots,n$ by induction we obtain that for any positive integer $k\geqslant2$ for the element $x_i^k$ of $\left(\mathbb{N}^n,\leqslant\right)$ the set ${\downarrow}x_i^k$ is a $k$-element chain, then so is $({\downarrow}x_i^k)\alpha$, which implies that only $i$-th coordinate of $(x_i^k)\alpha$ is equal to $k$ and all other coordinates are equal to $1$.

Hence we have shown that $\left(x_i^p\right)\alpha=x_i^p$ for any $i=1,\ldots,n$.

We observe that the converse map $\alpha^{-1}\colon \mathbb{N}^n\to \mathbb{N}^n$ to the order isomorphism $\alpha$ of $\left(\mathbb{N}^n,\leqslant\right)$ is an order isomorphism of $\left(\mathbb{N}^n,\leqslant\right)$ as well, which satisfies the assumption of the lemma. Hence the above part of the proof implies that $\left(x_i^p\right)\alpha^{-1}=x_i^p$ for any $i=1,\ldots,n$.

Fix an arbitrary $(x_1,\ldots,x_n)\in\mathbb{N}^n$ and suppose that
\begin{equation*}
(x_1,\ldots,x_n)\alpha=(y_1,\ldots,y_n).
\end{equation*}
Since a composition of monotone maps is again a monotone map, Lemma~\ref{lemma-2.2} and the above part of the proof imply that for any  $i=1,\ldots,n$ we have that
\begin{equation*}
\begin{split}
  (1,\ldots,\underbrace{x_i}_{i\hbox{\footnotesize{-th}}},\ldots,1) & =(1,\ldots,\underbrace{x_i}_{i\hbox{\footnotesize{-th}}},\ldots,1)\pi_i= ((1,\ldots,\underbrace{x_i}_{i\hbox{\footnotesize{-th}}},\ldots,1)\alpha)\pi_i
\leqslant\\
    & \leqslant ((x_1,\ldots,x_n)\alpha)\pi_i=(y_1,\ldots,y_n)\pi_i=(1,\ldots,\underbrace{y_i}_{i\hbox{\footnotesize{-th}}},\ldots,1)
\end{split}
\end{equation*}
and
\begin{equation*}
\begin{split}
  (1,\ldots,\underbrace{y_i}_{i\hbox{\footnotesize{-th}}},\ldots,1) & =(1,\ldots,\underbrace{y_i}_{i\hbox{\footnotesize{-th}}},\ldots,1)\pi_i= ((1,\ldots,\underbrace{y_i}_{i\hbox{\footnotesize{-th}}},\ldots,1)\alpha^{-1})\pi_i
\leqslant\\
    & \leqslant ((y_1,\ldots,y_n))\alpha^{-1}\pi_i=((x_1,\ldots,x_n)\alpha)\pi_i=(1,\ldots,\underbrace{x_i}_{i\hbox{\footnotesize{-th}}},\ldots,1).
\end{split}
\end{equation*}
This implies that $(x_1,\ldots,x_n)=(y_1,\ldots,y_n)$, which completes the proof of the lemma.
\end{proof}

\begin{theorem}\label{theorem-2.5}
For any positive integer $n$ the group of units $H(\mathbb{I})$ of the semigroup $\mathscr{I\!\!P\!F}(\mathbb{N}^n)$ is isomorphic to $\mathscr{S}_n$. Moreover, every element of $H(\mathbb{I})$ is induced by a permutation of the set $\{1,\ldots,n\}$.
\end{theorem}

\begin{proof}
We shall show that every order isomorphism of the poset $\left(\mathbb{N}^n,\leqslant\right)$ permutates coordinates of elements in $\mathbb{N}^n$. The converse statement follows from Lemma~\ref{lemma-2.3}.

We consider the element $x_1^2=(2,1,1,\ldots,1)$ of $\left(\mathbb{N}^n,\leqslant\right)$. Since ${\downarrow}x_1^2$ is a chain which consists of two elements and $\alpha$ is an order isomorphism of $\left(\mathbb{N}^n,\leqslant\right)$, the image $({\downarrow}x_1^2)\alpha$ is a two-element chain. By the definition of the poset $\left(\mathbb{N}^n,\leqslant\right)$, the element $(x_1^2)\alpha$ satisfies the following property: \emph{only one coordinate of $(x_1^2)\alpha$ is equal to $2$ and all other coordinates are equal to $1$.} Such coordinate of $(x_1^2)\alpha$ we denote by $\sigma_1$. Similar arguments show that for every element $x_i^2$ of the poset $\left(\mathbb{N}^n,\leqslant\right)$ we obtain that only $\sigma_i$-th coordinate of $(x_i^2)\alpha$ is equal to $2$ and all other coordinates are equal to $1$, for all $i=1,\ldots,n$. Also, since $\alpha$ is an order isomorphism of $\left(\mathbb{N}^n,\leqslant\right)$ we get that $\sigma_i=\sigma_j$ if and only if $i=j$, for $i,j=1,\ldots,n$.

Thus we defined a bijective map $\sigma\colon\mathbb{N}^n\to\mathbb{N}^n$ which acts on $\mathbb{N}^n$ as a permutation of coordinates. Then the compositions $\alpha\sigma^{-1}$ and $\sigma^{-1}\alpha$ are order isomorphisms of $\left(\mathbb{N}^n,\leqslant\right)$. Moreover $\alpha\sigma^{-1}$ and $\sigma^{-1}\alpha$ satisfy the assumption of Lemma~\ref{lemma-2.4}, which implies that the maps $\alpha\sigma^{-1}\colon \mathbb{N}^n\to \mathbb{N}^n$ and $\sigma^{-1}\alpha\colon \mathbb{N}^n\to \mathbb{N}^n$ are the identity maps of the set $\mathbb{N}^n$. This implies that $\alpha=\sigma$ which completes the proof the theorem.
\end{proof}

Theorems~2.3 and 2.20 from~\cite{Clifford-Preston-1961-1967} and Theorem~\ref{theorem-2.5} imply the following corollary.

\begin{corollary}\label{corollary-2.6}
Let $n$ be any positive integer. Then every maximal subgroup of $\mathscr{I\!\!P\!F}(\mathbb{N}^n)$ is isomorphic to $\mathscr{S}_n$ and every $\mathscr{H}$-class of $\mathscr{I\!\!P\!F}(\mathbb{N}^n)$ consists of $n!$ elements.
\end{corollary}

The following proposition gives sufficient conditions when an inverse subsemigroup of the semigroup $\mathscr{I\!\!P\!F}(\mathbb{N}^n)$, which is generated by an element of $\mathscr{I\!\!P\!F}(\mathbb{N}^n)$, is isomorphic to the bicyclic monoid ${\mathscr{C}}(p,q)$.

\begin{proposition}\label{proposition-2.7}
Let $n$ be any positive integer and $\alpha$ be an element of the semigroup $\mathscr{I\!\!P\!F}(\mathbb{N}^n)$ such that $\operatorname{ran}\alpha\subsetneqq \operatorname{dom}\alpha$. Then the inverse subsemigroup $\left\langle\alpha,\alpha^{-1}\right\rangle$ of $\mathscr{I\!\!P\!F}(\mathbb{N}^n)$, which is generated by $\alpha$, is isomorphic to the bicyclic monoid ${\mathscr{C}}(p,q)$.
\end{proposition}

\begin{proof}
Put $\varepsilon$ be the identity map of $\operatorname{dom}\alpha$. The semigroup operation of $\mathscr{I\!\!P\!F}(\mathbb{N}^n)$ implies that the following equalities hold:
\begin{equation*}
  \varepsilon\alpha=\alpha\varepsilon=\alpha, \qquad \varepsilon\alpha^{-1}=\alpha^{-1}\varepsilon=\alpha^{-1}, \qquad \alpha\alpha^{-1}=\varepsilon \qquad \hbox{and} \qquad \alpha^{-1}\alpha\neq\varepsilon
\end{equation*}
Then we apply Lemma~1.31 from \cite{Clifford-Preston-1961-1967}.
\end{proof}

\begin{corollary}\label{corollary-2.8}
Let $n$ be any positive integer. Then for any idempotents $\varepsilon$ and $\iota$ of the semigroup $\mathscr{I\!\!P\!F}(\mathbb{N}^n)$ such that $\varepsilon\preccurlyeq\iota$ there exists a subsemigroup ${\mathscr{C}}$ of $\mathscr{I\!\!P\!F}(\mathbb{N}^n)$ which is isomorphic to the bicyclic monoid ${\mathscr{C}}(p,q)$ and contains  $\varepsilon$ and $\iota$.
\end{corollary}

\begin{proof}
Suppose that $\varepsilon\neq\iota$. Let $\alpha$ be any order isomorphism from $\operatorname{dom}\iota$ onto $\operatorname{dom}\varepsilon$. Next we apply Proposition~\ref{proposition-2.7}.

If $\varepsilon=\iota$ then we choose any idempotent $\nu\neq\varepsilon$ such that $\nu\preccurlyeq\varepsilon$ and apply the above part of the proof.
\end{proof}

\begin{lemma}\label{lemma-2.9}
Let $n$ be any positive integer and $\mathfrak{C}$ be a congruence on the semigroup $\mathscr{I\!\!P\!F}(\mathbb{N}^n)$ such that $\varepsilon\mathfrak{C}\iota$ for some two distinct idempotents $\varepsilon,\iota\in\mathscr{I\!\!P\!F}(\mathbb{N}^n)$. Then $\varsigma\mathfrak{C}\upsilon$ for all idempotents $\varsigma,\upsilon$ of $\mathscr{I\!\!P\!F}(\mathbb{N}^n)$.
\end{lemma}

\begin{proof}
We observe that without loss of generality, we may assume that $\varepsilon\preccurlyeq\iota$ where $\preccurlyeq$ is the natural partial order on the semilattice $E(\mathscr{I\!\!P\!F}(\mathbb{N}^n))$. Indeed, if $\varepsilon,\iota\in E(\mathscr{I\!\!P\!F}(\mathbb{N}^n))$ then $\varepsilon\mathfrak{C}\iota$ implies that $\varepsilon=\varepsilon\varepsilon\mathfrak{C}\iota\varepsilon$, and since the idempotents $\varepsilon$ and $\iota$ are distinct in $\mathscr{I\!\!P\!F}(\mathbb{N}^n)$ we have that $\iota\varepsilon\preccurlyeq\varepsilon$.

Now, the inequality $\varepsilon\preccurlyeq\iota$ implies that $\operatorname{dom}\varepsilon\subseteq \operatorname{dom}\iota$ and hence  $\left(x_1,\ldots,x_n\right)\leqslant\left(y_1,\ldots,y_n\right)$, where ${\uparrow}\left(x_1,\ldots,x_n\right)$ and ${\uparrow}\left(y_1,\ldots,y_n\right)$ are principal filters in $\left(\mathbb{N}^n,\leqslant\right)$ such that
\begin{equation*}
{\uparrow}\left(x_1,\ldots,x_n\right)=\operatorname{dom}\iota \qquad \hbox{and} \qquad {\uparrow}\left(y_1,\ldots,y_n\right)=\operatorname{dom}\varepsilon.
\end{equation*}

Next, we define partial maps $\alpha,\beta\colon \mathbb{N}^n\rightharpoonup\mathbb{N}^n$ in the following way:
\begin{itemize}
  \item[$(a)$] $\operatorname{dom}\alpha=\mathbb{N}^n$, $\operatorname{ran}\alpha=\operatorname{dom}\iota$ and $\left(z_1,\ldots,z_n\right)\alpha=\left(z_1+x_1-1,\ldots,z_n+x_n-1\right)$ for $\left(z_1,\ldots,z_n\right)\in\operatorname{dom}\alpha$;
  \item[$(b)$] $\operatorname{dom}\beta=\operatorname{dom}\iota$, $\operatorname{ran}\beta=\mathbb{N}^n$ and $\left(z_1,\ldots,z_n\right)\beta=\left(z_1-x_1+1,\ldots,z_n-x_n+1\right)$ for $\left(z_1,\ldots,z_n\right)\in\operatorname{dom}\beta$.
\end{itemize}
Simple verifications show that $\alpha\iota\beta=\mathbb{I}$ and $\beta\alpha=\iota$, and moreover since $\alpha\beta=\mathbb{I}$ we have that
\begin{equation*}
\left(\alpha\varepsilon\beta\right)\left(\alpha\varepsilon\beta\right)=\alpha\varepsilon\left(\beta\alpha\right)\varepsilon\beta= \alpha\varepsilon\iota\varepsilon\beta=\alpha\varepsilon\varepsilon\beta=\alpha\varepsilon\beta,
\end{equation*}
which implies that $\alpha\varepsilon\beta$ is an idempotent of $\mathscr{I\!\!P\!F}(\mathbb{N}^n)$ such that  $\alpha\varepsilon\beta\neq \mathbb{I}$.

Thus, it was shown that there exists a non-unit idempotent $\varepsilon^*$ in $\mathscr{I\!\!P\!F}(\mathbb{N}^n)$ such that $\varepsilon^*\mathfrak{C}\mathbb{I}$. This implies that $\varepsilon_0\mathfrak{C}\mathbb{I}$ for any idempotent $\varepsilon_0$ of $\mathscr{I\!\!P\!F}(\mathbb{N}^n)$ such that $\varepsilon^*\preccurlyeq\varepsilon_0\preccurlyeq\mathbb{I}$. Then the definition of the semigroup $\mathscr{I\!\!P\!F}(\mathbb{N}^n)$ and Proposition~\ref{proposition-2.1}$(ii)$ imply that there exists an element $x_i^2=(1,\ldots,\underbrace{{2}}_{i\hbox{\footnotesize{-th}}},\ldots,1)$ of the poset $\left(\mathbb{N}^n,\leqslant\right)$  and there exists an idempotent $\varepsilon_i$ in $\mathscr{I\!\!P\!F}(\mathbb{N}^n)$ such that $\operatorname{dom}\varepsilon_0\subseteq\operatorname{dom}\varepsilon_i={\uparrow}x_i^2$.

Fix an arbitrary positive integer $j\in\left\{1,\ldots,n\right\}\setminus\{i\}$. Put $\sigma_{(i,j)}$ is the permutation of coordinates of elements of the set $\mathbb{N}^n$ which permutates only $j$-th and $i$-th coordinates, i.e., this permutation is the cycle  $(i,j)$ on the coordinates. Then the semigroup operation of $\mathscr{I\!\!P\!F}(\mathbb{N}^n)$ implies that $\sigma_{(i,j)}\mathbb{I}\sigma_{(i,j)}=\mathbb{I}$ and $\sigma_{(i,j)}\varepsilon_i\sigma_{(i,j)}=\varepsilon_j$, where $\varepsilon_j$ is the identity map of the principal filter ${\uparrow}x_j^2$ of the poset $\left(\mathbb{N}^n,\leqslant\right)$.

The above arguments imply that $\varepsilon_k\mathfrak{C}\mathbb{I}$ for every idempotent $\varepsilon_k\in\mathscr{I\!\!P\!F}(\mathbb{N}^n)$ such that $\varepsilon_k$ is the identity map of the principal filter ${\uparrow}x_k^2$ of the poset $\left(\mathbb{N}^n,\leqslant\right)$, $k=1,\ldots,n$.  The semigroup operation of $\mathscr{I\!\!P\!F}(\mathbb{N}^n)$ implies that the idempotent $\varepsilon_1\ldots\varepsilon_n$ is the identity map of the principal filter ${\uparrow}(2,\ldots,2)$ of $\left(\mathbb{N}^n,\leqslant\right)$. Since $\varepsilon^*\mathfrak{C}\mathbb{I}$ for a some non-unit idempotent $\varepsilon^*$ in $\mathscr{I\!\!P\!F}(\mathbb{N}^n)$, there exists an idempotent $\varepsilon_k\in\mathscr{I\!\!P\!F}(\mathbb{N}^n)$ such that $\varepsilon_k\mathfrak{C}\mathbb{I}$. Then by above part of the proof we get that $\mathbb{I}\mathfrak{C}\left(\varepsilon_1\ldots\varepsilon_n\right)$. We define a partial map $\gamma\colon \mathbb{N}^n\rightharpoonup \mathbb{N}^n$ in the following way:
\begin{equation*}
  \operatorname{dom}\gamma=\mathbb{N}^n, \quad \operatorname{ran}\gamma={\uparrow}(2,\ldots,2) \quad  \hbox{~and~} \quad \left(z_1,\ldots,z_n\right)\gamma=\left(z_1+1,\ldots,z_n+1\right),
\end{equation*}
for $\left(z_1,\ldots,z_n\right)\in\operatorname{dom}\gamma$.
By Proposition~\ref{proposition-2.7}, the subsemigroup $\left\langle\gamma,\gamma^{-1}\right\rangle$ of $\mathscr{I\!\!P\!F}(\mathbb{N}^n)$, which is generated by $\gamma$  and its inverse $\gamma^{-1}$, is isomorphic to the bicyclic monoid ${\mathscr{C}}(p,q)$. It is obvious that $\gamma\gamma^{-1}=\mathbb{I}$ and $\gamma^{-1}\gamma=\varepsilon_1\ldots\varepsilon_n$. Since $\mathbb{I}\mathfrak{C}\left(\varepsilon_1\ldots\varepsilon_n\right)$, by Corollary~1.32 from \cite{Clifford-Preston-1961-1967} we obtain that all idempotents of the subsemigroup $\left\langle\gamma,\gamma^{-1}\right\rangle$ in $\mathscr{I\!\!P\!F}(\mathbb{N}^n)$ are $\mathfrak{C}$-equivalent. Also, the definition of the bicyclic semigroup ${\mathscr{C}}(p,q)$ and Lemma~1.31 from \cite{Clifford-Preston-1961-1967} imply that all idempotents of the subsemigroup $\left\langle\gamma,\gamma^{-1}\right\rangle$ of $\mathscr{I\!\!P\!F}(\mathbb{N}^n)$ are elements of the form $\left(\gamma^{-1}\right)^k\gamma^k$, where $k$ is a some non-negative integer. Now, by the definition of the semigroup $\mathscr{I\!\!P\!F}(\mathbb{N}^n)$ we have that $\left(\gamma^{-1}\right)^k\gamma^k$ is the identity map of the principal filter ${\uparrow}(k,\ldots,k)$ of $\left(\mathbb{N}^n,\leqslant\right)$ for some  non-negative integer $k$. Moreover, for every idempotent $\zeta$ of $\mathscr{I\!\!P\!F}(\mathbb{N}^n)$ which is the identity map of the principal filter ${\uparrow}(a_1,\ldots,a_n)$ of $\left(\mathbb{N}^n,\leqslant\right)$, we have that $\left(\gamma^{-1}\right)^m\gamma^m\preccurlyeq\zeta$, where
$
  m=\max\left\{a_1,\ldots,a_n\right\},
$
which implies that $\mathbb{I}\mathfrak{C}\zeta$.
\end{proof}

\begin{lemma}\label{lemma-2.10}
Let $n$ be any positive integer  and $\mathfrak{C}$ be a congruence on the semigroup $\mathscr{I\!\!P\!F}(\mathbb{N}^n)$ such that $\alpha\mathfrak{C}\beta$ for some non-$\mathscr{H}$-equivalent elements $\alpha,\beta\in\mathscr{I\!\!P\!F}(\mathbb{N}^n)$. Then $\varepsilon\mathfrak{C}\iota$ for all idempotents $\varepsilon,\iota$ of $\mathscr{I\!\!P\!F}(\mathbb{N}^n)$.
\end{lemma}

\begin{proof}
Since $\alpha$ and $\beta$ are not $\mathscr{H}$-equivalent in $\mathscr{I\!\!P\!F}(\mathbb{N}^n)$ we have that either $\alpha\alpha^{-1}\neq\beta\beta^{-1}$ or $\alpha^{-1}\alpha\neq\beta^{-1}\beta$ (see \cite[p.~82]{Lawson-1998}). Then Proposition~4 from \cite[Section~2.3]{Lawson-1998} implies that  $\alpha\alpha^{-1}\mathfrak{C}\beta\beta^{-1}$ and $\alpha^{-1}\alpha\mathfrak{C}\beta^{-1}\beta$ and hence the assumption of Lemma~\ref{lemma-2.9} holds.
\end{proof}

\begin{lemma}\label{lemma-2.11}
Let $n$ be any positive integer  and $\mathfrak{C}$ be a congruence on the semigroup $\mathscr{I\!\!P\!F}(\mathbb{N}^n)$ such that $\alpha\mathfrak{C}\beta$ for some two distinct $\mathscr{H}$-equivalent elements $\alpha,\beta\in\mathscr{I\!\!P\!F}(\mathbb{N}^n)$. Then $\varepsilon\mathfrak{C}\iota$ for all idempotents $\varepsilon,\iota$ of $\mathscr{I\!\!P\!F}(\mathbb{N}^n)$.
\end{lemma}

\begin{proof}
By Proposition~\ref{proposition-2.1}$(vii)$ the semigroup $\mathscr{I\!\!P\!F}(\mathbb{N}^n)$ is simple and then Theorem~2.3 from \cite{Clifford-Preston-1961-1967} implies that there exist $\mu,\xi\in\mathscr{I\!\!P\!F}(\mathbb{N}^n)$ such that $f\colon H_{\alpha}\to H_{\mathbb{I}}\colon \chi\mapsto\mu\chi\xi$ maps $\alpha$ to $\mathbb{I}$ and $\beta$ to $\gamma\neq\mathbb{I}$, respectively, which implies that $\mathbb{I}\mathfrak{C}\gamma$. Since $\gamma\neq\mathbb{I}$ is an element of the group of units of the semigroup $\mathscr{I\!\!P\!F}(\mathbb{N}^n)$, by Theorem~\ref{theorem-2.5}, $\gamma$ permutates coordinates of elements of $\mathbb{N}^n$, and hence there exists a positive integer $i_\gamma$ such that $\left(x_i^2\right)\gamma\neq x_{i_\gamma}^2$. Also, by Theorem~\ref{theorem-2.5}, there exists a positive integer $j_\gamma\in\left\{1,\ldots,n\right\}\setminus\left\{i_\gamma\right\}$ such that $\big(x_{i_\gamma}^2\big)\gamma=x_{j_\gamma}^2=(1,\ldots,\underbrace{{2}}_{{j_\gamma}\hbox{\footnotesize{-th}}},\ldots,1)$ is the element of the poset $\left(\mathbb{N}^n,\leqslant\right)$.

By $\varepsilon$ we denote the identity map of the principal filter ${\uparrow}x_{i_\gamma}^2$. Since $\mathfrak{C}$ is a congruence on the semigroup $\mathscr{I\!\!P\!F}(\mathbb{N}^n)$ and $\gamma\in H_{\mathbb{I}}$ we have that
\begin{equation*}
  \varepsilon=\varepsilon\varepsilon=\varepsilon\mathbb{I}\varepsilon\mathfrak{C}\varepsilon\gamma\varepsilon.
\end{equation*}
Since ${j_\gamma}\neq i_\gamma$, the semigroup operation of $\mathscr{I\!\!P\!F}(\mathbb{N}^n)$ implies that $\operatorname{dom}(\varepsilon\gamma\varepsilon) \subsetneqq\operatorname{dom}\varepsilon$. Then by Proposition~\ref{proposition-2.1}$(v)$, $\varepsilon\gamma\varepsilon$ and $\varepsilon$ are non-$\mathscr{H}$-equivalent elements in $\mathscr{I\!\!P\!F}(\mathbb{N}^n)$. Next we apply Lemma~\ref{lemma-2.10}.
\end{proof}

\begin{theorem}\label{theorem-2.12}
Let $n$ be any positive integer. Then every non-identity congruence $\mathfrak{C}$ on the semigroup $\mathscr{I\!\!P\!F}(\mathbb{N}^n)$ is a group congruence.
\end{theorem}

\begin{proof}
For every non-identity congruence $\mathfrak{C}$ on  $\mathscr{I\!\!P\!F}(\mathbb{N}^n)$ there exist two distinct elements $\alpha,\beta\in\mathscr{I\!\!P\!F}(\mathbb{N}^n)$ such that $\alpha\mathfrak{C}\beta$. If $\alpha\mathscr{H}\beta$ in $\mathscr{I\!\!P\!F}(\mathbb{N}^n)$ then, by Lemma~\ref{lemma-2.11}, all idempotents of the semigroup $\mathscr{I\!\!P\!F}(\mathbb{N}^n)$ are $\mathfrak{C}$-equivalent, otherwise by Lemma~\ref{lemma-2.10} we get the same. Thus, by Lemma~II.1.10 from \cite{Petrich-1984}, the quotient semigroup $\mathscr{I\!\!P\!F}(\mathbb{N}^n)/\mathfrak{C}$ has a unique idempotent and hence it is a group.
\end{proof}

For arbitrary elements $\mathbf{x}=(x_1,\ldots,x_n)$ and $\mathbf{y}=(y_1,\ldots,y_n)$ of $\mathbb{N}^n$ and any permutation
$\sigma\colon\left\{1,\ldots,n\right\}$ $\to\left\{1,\ldots,n\right\}$   we denote
\begin{equation*}
\begin{split}
(\mathbf{x})\sigma &=\left(x_{(1)\sigma^{-1}},\ldots,x_{(n)\sigma^{-1}}\right);\\
\max\{\mathbf{x},\mathbf{y}\}&=\left(\max\{x_1,y_1\},\ldots,\max\{x_n,y_n\}\right).
\end{split}
\end{equation*}

\begin{remark}
Since elements of $\mathbb N^n$ are functions $\{1,\ldots,n\}\to \mathbb N$, the action of $\sigma$ on such $\mathbf{x}$ should be written as $\sigma(\mathbf{x})$ according the the composition rules accepted by the authors and classic traditions. But following the classical traditions of the theory of semigroups of transformations, instead of $\sigma(\mathbf{x})$, we shall write $(\mathbf{x})\sigma$, because the permutation $\sigma$ we consider as a transformation of $\mathbb N^n$. By the way, $(x)\sigma$ is exactly $x\circ\alpha$ in the standard sense.
\end{remark}

\begin{lemma}\label{lemma-2.13}
For every positive integer $n$ and any $\mathbf{x},\mathbf{y} \in\mathbb{N}^n$ the following conditions hold:
\begin{itemize}
  \item[$(i)$] $(\mathbf{x}+\mathbf{y})\sigma=(\mathbf{x})\sigma+(\mathbf{y})\sigma$;
  \item[$(ii)$] $(\mathbf{x}-\mathbf{y})\sigma=(\mathbf{x})\sigma-(\mathbf{y})\sigma$ in the case when $\mathbf{y}\leqslant \mathbf{x}$;
  \item[$(iii)$] $(\max\{\mathbf{x},\mathbf{y}\})\sigma=\max\{(\mathbf{x})\sigma,(\mathbf{y})\sigma\}$.
\end{itemize}
\end{lemma}

\begin{proof}
In $(i)$ we have that
\begin{equation*}
  (\mathbf{x}+\mathbf{y})\sigma=(x_1+y_1,\ldots,x_n+y_n) \sigma =(p_1,\ldots,p_n)\sigma
\end{equation*}
for $p_i=x_i+y_i$, $i=1,\ldots,n$, and then
\begin{equation*}
(p_1,\ldots,p_n)\sigma =(p_{(1)\sigma^{-1}},\ldots,p_{(n)\sigma^{-1}})=(x_{(1)\sigma^{-1}}+y_{(1)\sigma^{-1}},\ldots,x_{(n)\sigma^{-1}}+y_{(n)\sigma^{-1}})=
(\mathbf{x})\sigma+(\mathbf{y})\sigma.
\end{equation*}

The proofs of $(ii)$ and $(iii)$ are similar.
\end{proof}

The statement of the following lemma follows from the definition of the semigroup $\mathscr{I\!\!P\!F}(\mathbb{N}^n)$.

\begin{lemma}\label{lemma-2.14}
For every $\alpha\in\mathscr{I\!\!P\!F}(\mathbb{N}^n)$ there exist unique $\mathbf{x},\mathbf{y} \in\mathbb{N}^n$,
$\rho_\alpha,\lambda_\alpha\in\mathscr{I\!\!P\!F}(\mathbb{N}^n)$ and $ \sigma_\alpha \in \mathscr{S}_n$ such that ${\alpha=\rho_\alpha\sigma_\alpha\lambda_\alpha}$ and
\begin{equation*}
\begin{split}
\operatorname{dom}\rho_\alpha&=\operatorname{dom}\alpha={\uparrow}\mathbf{x}, \quad \operatorname{ran}\rho_\alpha=\mathbb{N}^n, \quad (\mathbf{z})\rho_\alpha=\mathbf{z}-\mathbf{x}+\mathbf{1}  \quad \hbox{ for } \; \mathbf{z}\in \operatorname{dom}\rho_\alpha; \\
\operatorname{ran}\lambda_\alpha&=\operatorname{ran}\alpha={\uparrow}\mathbf{y}, \quad \operatorname{dom}\lambda_\alpha=\mathbb{N}^n, \quad  (\mathbf{z})\lambda_\alpha=\mathbf{z}+\mathbf{y}-\mathbf{1}  \quad  \hbox{ for } \;  \mathbf{z}\in \operatorname{dom}\lambda_\alpha,
\end{split}
\end{equation*}
where $\mathbf{1}=(1,\ldots,1)$ is the smallest element of the poset $\left(\mathbb{N}^n,\leqslant\right)$.
\end{lemma}

Later in this section, for every $\alpha\in\mathscr{I\!\!P\!F}(\mathbb{N}^n)$ by $\rho_\alpha, \lambda_\alpha$ and $\sigma_\alpha$ we denote the elements $\rho_\alpha,\lambda_\alpha\in\mathscr{I\!\!P\!F}(\mathbb{N}^n)$ and $ \sigma_\alpha \in \mathscr{S}_n$ whose are determined in Lemma~\ref{lemma-2.14}.

\begin{lemma}\label{lemma-2.15}
Let $\alpha$ and $\beta$ be elements of the semigroup $\mathscr{I\!\!P\!F}(\mathbb{N}^n)$ such that  $\operatorname{dom}\alpha={\uparrow}\mathbf{x}$, $\operatorname{ran}\alpha={\uparrow}\mathbf{y}$,  $\operatorname{dom}\beta={\uparrow}\mathbf{u}$ and $\operatorname{ran}\beta={\uparrow}\mathbf{v}$. Then
\begin{equation*}
\begin{split}
\operatorname{dom}(\alpha\beta)&={\uparrow}[(\max\{\mathbf{y},\mathbf{u}\}-\mathbf{y})\sigma_\alpha^{-1}+\mathbf{x}];\\
\operatorname{ran}(\alpha\beta)&={\uparrow}[(\max\{\mathbf{y},\mathbf{u}\}-\mathbf{u})\sigma_\beta+\mathbf{v}];\\
\sigma_{\alpha\beta}&=\sigma_\alpha\sigma_\beta.
\end{split}
\end{equation*}
\end{lemma}

\begin{proof}
The definition of the domain of the composition of partial transformations (see \cite[p.~4]{Lawson-1998}) implies that
\begin{equation*}
\operatorname{dom}(\alpha\beta)=[\operatorname{ran}\alpha\cap \operatorname{dom}\beta]\alpha^{-1}
=[{\uparrow}\mathbf{y}\cap {\uparrow}\mathbf{u}]\alpha^{-1}
= [{\uparrow}\max\{ \mathbf{y},\mathbf{u}\}]\alpha^{-1}.
\end{equation*}
Since $\alpha$ is a monotone bijection between principal filters of the poset $\left(\mathbb{N}^n,\leqslant\right)$ we get that
\begin{equation*}
  [{\uparrow}\max\{ \mathbf{y},\mathbf{u}\}]\alpha^{-1}= {\uparrow}\left([\max\{ \mathbf{y},\mathbf{u}\}]\alpha^{-1}\right),
\end{equation*}
and, by Lemma~\ref{lemma-2.14},
\begin{equation*}
\begin{split}
\operatorname{dom}(\alpha\beta)&={\uparrow}\left([\max\{ \mathbf{y},\mathbf{u}\}]\alpha^{-1}\right)=\\
&= {\uparrow}\left([\max\{ \mathbf{y},\mathbf{u}\}]\lambda_\alpha^{-1}\sigma_\alpha^{-1}\rho_\alpha^{-1}\right)=\\
&= {\uparrow}\left([\max\{ \mathbf{y},\mathbf{u}\}-\mathbf{y}+\mathbf{1}]\sigma_\alpha^{-1}\rho_\alpha^{-1}\right)=\\
&= {\uparrow}\left([(\max\{ \mathbf{y},\mathbf{u}\}-\mathbf{y})\sigma_\alpha^{-1}+\mathbf{1}]\rho_\alpha^{-1}\right)=\\
&= {\uparrow}[(\max\{ \mathbf{y},\mathbf{u}\}-\mathbf{y})\sigma_\alpha^{-1}+\mathbf{x}].
\end{split}
\end{equation*}

Similarly, the definition of the range of the composition of partial transformations (see \cite[p.~4]{Lawson-1998}) implies that
\begin{equation*}
\operatorname{ran}(\alpha\beta)=[\operatorname{ran}\alpha\cap \operatorname{dom}\beta]\beta
=[{\uparrow}\mathbf{y}\cap {\uparrow}\mathbf{u}]\beta
= [{\uparrow}\max\{ \mathbf{y},\mathbf{u}\}]\beta.
\end{equation*}
Since $\beta$ is a monotone bijection between principal filters of the poset $\left(\mathbb{N}^n,\leqslant\right)$ we get that
\begin{equation*}
  [{\uparrow}\max\{ \mathbf{y},\mathbf{u}\}]\beta= {\uparrow}\left([\max\{ \mathbf{y},\mathbf{u}\}]\beta\right),
\end{equation*}
and, by Lemma~\ref{lemma-2.14},
\begin{equation*}
\begin{split}
\operatorname{ran}(\alpha\beta)&={\uparrow}\left([\max\{ \mathbf{y},\mathbf{u}\}]\beta\right)=\\
&= {\uparrow}\left([\max\{ \mathbf{y},\mathbf{u}\}]\rho_\beta\sigma_\beta\lambda_\beta\right)=\\
&= {\uparrow}\left([\max\{ \mathbf{y},\mathbf{u}\}-\mathbf{u}+\mathbf{1}]\sigma_\beta\lambda_\beta\right)=\\
&= {\uparrow}\left([(\max\{ \mathbf{y},\mathbf{u}\}-\mathbf{u})\sigma_\beta+\mathbf{1}]\lambda_\beta\right)=\\
&= {\uparrow}[(\max\{\mathbf{y},\mathbf{u}\}-\mathbf{u})\sigma_\beta+\mathbf{v}].
\end{split}
\end{equation*}

We observe that definitions of elements $\sigma_\alpha$ and $\sigma_\beta$ imply that
\begin{equation*}
\operatorname{dom}\sigma_\alpha=\operatorname{ran}\sigma_\alpha=\operatorname{dom}\sigma_\beta=\operatorname{ran}\sigma_\beta=\mathbb{N}^n,
\end{equation*}
and hence $\operatorname{dom}(\sigma_\alpha\sigma_\beta)=\operatorname{ran}(\sigma_\alpha\sigma_\beta)=\mathbb{N}^n$. Since $\rho_\alpha,\lambda_\alpha,\sigma_\alpha,\rho_\beta,\lambda_\beta,\sigma_\beta$ are partial bijection of $\mathbb{N}^n$ and $\operatorname{dom}(\alpha\beta)=\operatorname{dom}(\rho_{\alpha\beta}\sigma_\alpha\sigma_\beta\lambda_{\alpha\beta})$, the equality $\alpha\beta=\rho_{\alpha\beta}\sigma_\alpha\sigma_\beta\lambda_{\alpha\beta}$ implies that $\sigma_{\alpha\beta}=\sigma_\alpha\sigma_\beta$.

Next we shall show that the equality $\alpha\beta=\rho_{\alpha\beta}\sigma_\alpha\sigma_\beta\lambda_{\alpha\beta}$ holds.
We observe that for any $\mathbf{z}\in\operatorname{dom}(\alpha\beta)$ there exists a unique $\mathbf{p}\in\mathbb{N}^n\cup\{(0,\ldots,0),(1,\ldots,0),\ldots,(0,\ldots,1)\}$ such that
\begin{equation*}
\mathbf{z}=(\max\{ \mathbf{y},\mathbf{u}\}-\mathbf{y})\sigma_\alpha^{-1}+\mathbf{x}+\mathbf{p}.
\end{equation*}
Then we have that
\begin{equation*}
\begin{split}
\left(\mathbf{z}\right)\alpha\beta&=\left((\max\{ \mathbf{y},\mathbf{u}\}-\mathbf{y})\sigma_\alpha^{-1}+\mathbf{x}+\mathbf{p}\right)\alpha\beta=\\
&=\left((\max\{ \mathbf{y},\mathbf{u}\}-\mathbf{y})\sigma_\alpha^{-1}+\mathbf{x}+\mathbf{p}\right)\rho_\alpha\sigma_\alpha\lambda_\alpha\beta=\\
&=\left((\max\{ \mathbf{y},\mathbf{u}\}-\mathbf{y})\sigma_\alpha^{-1}+\mathbf{p}+\mathbf{1}\right)\sigma_\alpha\lambda_\alpha\beta=\\
&=\left(\max\{ \mathbf{y},\mathbf{u}\}-\mathbf{y}+{(\mathbf{p})\sigma_\alpha}+\mathbf{1}\right)\lambda_\alpha\beta=\\
&=\left(\max\{ \mathbf{y},\mathbf{u}\}+{(\mathbf{p})\sigma_\alpha}\right)\rho_\beta\sigma_\beta\lambda_\beta=\\
&=\left(\max\{ \mathbf{y},\mathbf{u}\}-\mathbf{u}+{(\mathbf{p})\sigma_\alpha}+\mathbf{1}\right)\sigma_\beta\lambda_\beta=\\
&=\left((\max\{ \mathbf{y},\mathbf{u}\}-\mathbf{u})\sigma_\beta+{((\mathbf{p})\sigma_\alpha)\sigma_\beta}+\mathbf{1}\right)\lambda_\beta=\\
&=(\max\{ \mathbf{y},\mathbf{u}\}-\mathbf{u})\sigma_\beta+{((\mathbf{p})\sigma_\alpha)\sigma_\beta}+\mathbf{v}
\end{split}
\end{equation*}
and
\begin{equation*}
\begin{split}
(\mathbf{z})\rho_{\alpha\beta}\sigma_\alpha\sigma_\beta\lambda_{\alpha\beta}
&=\left((\max\{\mathbf{y},\mathbf{u}\}-\mathbf{y})\sigma_\alpha^{-1}+\mathbf{x}+\mathbf{p}\right)\rho_{\alpha\beta}\sigma_\alpha\sigma_\beta\lambda_{\alpha\beta}=\\
&=\left(\mathbf{p}+1\right)\sigma_\alpha\sigma_\beta\lambda_{\alpha\beta}=\\
&=\left(((\mathbf{p})\sigma_\alpha)\sigma_\beta+\mathbf{1}\right)\lambda_{\alpha\beta}=\\
&=(\max\{ \mathbf{y},\mathbf{u}\}-\mathbf{u})\sigma_\beta+\mathbf{v}+{((\mathbf{p})\sigma_\alpha)\sigma_\beta}).
\end{split}
\end{equation*}
This completes the proof of the lemma.
\end{proof}

\begin{proposition}\label{proposition-2.16}
Let $\alpha$ and $\beta$ be elements of the semigroup $\mathscr{I\!\!P\!F}(\mathbb{N}^n)$ such that  $\operatorname{dom}\alpha={\uparrow}\mathbf{x}$, $\operatorname{ran}\alpha={\uparrow}\mathbf{y}$,  $\operatorname{dom}\beta={\uparrow}\mathbf{u}$ and $\operatorname{ran}\beta={\uparrow}\mathbf{v}$. Then the following statements hold:
\begin{itemize}
  \item[$(i)$] $\alpha$ is an idempotent of $\mathscr{I\!\!P\!F}(\mathbb{N}^n)$ if and only if $\lambda_\alpha$ is an inverse partial map to $\rho_\alpha$, i.e., $\mathbf{x}=\mathbf{y}$, and $\sigma_\alpha$ is the identity element of the group $\mathscr{S}_n$;
  \item[$(ii)$] $\alpha$ is inverse of $\beta$ in $\mathscr{I\!\!P\!F}(\mathbb{N}^n)$ if and only if $\mathbf{x}=\mathbf{v}$, $\mathbf{y}=\mathbf{u}$ (i.e., $\lambda_\alpha$ is an inverse partial map to $\rho_\beta$ and $\lambda_\beta$ is an inverse partial map to $\rho_\alpha$) and $\sigma_\alpha$ is inverse of $\sigma_\beta$ in the group $\mathscr{S}_n$.
\end{itemize}
\end{proposition}

\begin{proof}
$(i)$ Suppose that $\alpha$ is an idempotent of $\mathscr{I\!\!P\!F}(\mathbb{N}^n)$. Since $\alpha$ is an identity map of a some principal filter of  $\left(\mathbb{N}^n,\leqslant\right)$, we have that $\mathbf{x}=\mathbf{y}$ and hence $\lambda_\alpha$ is the converse partial map to $\rho_\alpha$. Then the equalities
\begin{equation*}
  \rho_\alpha\sigma_\alpha\lambda_\alpha=\alpha=\alpha\alpha=\rho_\alpha\sigma_\alpha\lambda_\alpha\rho_\alpha\sigma_\alpha\lambda_\alpha= \rho_\alpha\sigma_\alpha\sigma_\alpha\lambda_\alpha
\end{equation*}
and Lemma~\ref{lemma-2.14} imply that $\sigma_\alpha=\sigma_\alpha\sigma_\alpha$, and hence $\sigma_\alpha$ is the identity element of the group $\mathscr{S}_n$.

The converse statement is obvious.

$(ii)$ Suppose that $\alpha$ and $\beta$ are inverse elements in $\mathscr{I\!\!P\!F}(\mathbb{N}^n)$. Then $\operatorname{dom}\alpha=\operatorname{ran}\beta$ and $\operatorname{ran}\alpha=\operatorname{dom}\beta$ and hence we get that $\mathbf{x}=\mathbf{v}$ and $\mathbf{y}=\mathbf{u}$. This and Lemma~\ref{lemma-2.14} imply that $\lambda_\alpha$ is the converse partial map to $\rho_\beta$ and $\lambda_\beta$ is an inverse partial map to $\rho_\alpha$. Since $\alpha\beta$ is an idempotent of $\mathscr{I\!\!P\!F}(\mathbb{N}^n)$ the above arguments imply that
\begin{equation*}
  \alpha\beta=\rho_\alpha\sigma_\alpha\lambda_\alpha\rho_\beta\sigma_\beta\lambda_\beta=\rho_\alpha\sigma_\alpha\sigma_\beta\lambda_\beta,
\end{equation*}
and hence by statement $(i)$ the element $\sigma_\alpha\sigma_\beta$ is the identity of the group $\mathscr{S}_n$. This implies that $\sigma_\alpha$ is inverse of $\sigma_\beta$ in $\mathscr{S}_n$.

The converse statement is obvious.
\end{proof}

\begin{remark}\label{remark-2.17}
In the bicyclic semigroup ${\mathscr{C}}(p,q)$ the semigroup operation is determined in the following way:
\begin{equation*}
  p^iq^j\cdot p^kq^l=
\left\{
  \begin{array}{ll}
    p^iq^{j-k+l}, & \hbox{if~} j>k;\\
    p^iq^l,       & \hbox{if~} j=k;\\
    p^{i-j+k}q^l, & \hbox{if~} j<k,
  \end{array}
\right.
\end{equation*}
which is equivalent to the following formula:
\begin{equation*}
  p^iq^j\cdot p^kq^l=p^{i+\max\{j,k\}-j}q^{l+\max\{j,k\}-k}.
\end{equation*}
The above implies that the bicyclic semigroup ${\mathscr{C}}(p,q)$ is isomorphic to the semigroup $(S,*)$ which is defined on the square $\mathbb{N}_0\times \mathbb{N}_0$ of the the set of non-negative integers with the following multiplication:
\begin{equation}\label{eq-2.1}
  (i,j)*(k,l)=(i+\max\{j,k\}-j,l+\max\{j,k\}-k).
\end{equation}
\end{remark}

Later, for an arbitrary positive integer $n$ by ${\mathscr{C}}(p,q)^n$ we shall denote the $n$-th direct power of $(S,*)$, i.e., ${\mathscr{C}}(p,q)^n$ is the $n$-th power of $\mathbb{N}_0\times \mathbb{N}_0$ with the point-wise semigroup operation defined by formula~\eqref{eq-2.1}. Also, by $[\mathbf{x},\mathbf{y}]$ we denote the ordered collection $\left((x_1,y_1),\dots,(x_n,y_n)\right)$ of ${\mathscr{C}}(p,q)^n$, where $\mathbf{x}=(x_1,\ldots,x_n)$ and $\mathbf{y}=(y_1,\ldots,y_n)$, and for arbitrary permutation $\sigma\colon\left\{1,\ldots,n\right\}\to\left\{1,\ldots,n\right\}$  we put
\begin{equation*}\label{eq-2.2}
(\mathbf{x})\sigma =\left(x_{(1)\sigma^{-1}},\ldots,x_{(n)\sigma^{-1}}\right).
\end{equation*}

Let $\operatorname{\mathbf{Aut}}({{\mathscr{C}^n}(p,q)})$ be the automorphism group of  the semigroup ${\mathscr{C}}(p,q)^n$. We define a map $\Phi$ from $\mathscr{S}_n$ into all selfmaps of the semigroup ${\mathscr{C}}(p,q)^n$ putting $\sigma\mapsto \Phi_{\sigma}$, where the map $\Phi_{\sigma}\colon {{\mathscr{C}^n}(p,q)}\to {{\mathscr{C}^n}(p,q)}$ is defined by the formula:
\begin{equation}\label{eq-2.3}
  ([\mathbf{x},\mathbf{y}])\Phi_{\sigma}=[(\mathbf{x})\sigma,(\mathbf{y})\sigma].
\end{equation}
It is obvious that the map $\Phi_{\sigma}$ is a bijection of ${\mathscr{C}}(p,q)^n$ and $\Phi_{\sigma_1}\neq\Phi_{\sigma_2}$ for distinct $\sigma_1,\sigma_2\in\mathscr{S}_n$.

Since
\begin{equation*}
\begin{split}
([\mathbf{x},\mathbf{y}]*[\mathbf{u},\mathbf{v}])\Phi_\sigma=&\left([\max\{\mathbf{y},\mathbf{u}\}-\mathbf{y}+\mathbf{x}, \max\{\mathbf{y},\mathbf{u}\}-\mathbf{u}+\mathbf{v}]\right)\Phi_\sigma=\\
=&\left[(\max\{\mathbf{y},\mathbf{u}\}-\mathbf{y}+\mathbf{x})\sigma,(\max\{\mathbf{y},\mathbf{u}\}-\mathbf{u}+\mathbf{v})\sigma\right]=\\
=&\left[\max\{(\mathbf{y})\sigma,(\mathbf{u})\sigma\}-(\mathbf{y})\sigma+(\mathbf{x})\sigma, \max\{(\mathbf{y})\sigma,(\mathbf{u})\sigma\}-(\mathbf{u})\sigma+(\mathbf{v})\sigma\right]=\\
=&\left[(\mathbf{x})\sigma,(\mathbf{y})\sigma]*[(\mathbf{u})\sigma,(\mathbf{v})\sigma\right]=\\
=&\left[\mathbf{x},\mathbf{y}\right]\Phi_\sigma*\left[\mathbf{u},\mathbf{v}\right]\Phi_\sigma
\end{split}
\end{equation*}
and
\begin{equation*}
  ([\mathbf{x},\mathbf{y}])\Phi_{\sigma_1\sigma_2}=[(\mathbf{x})\sigma_1\sigma_2,(\mathbf{y})\sigma_1\sigma_2]=
([(\mathbf{x})\sigma_1,(\mathbf{y})\sigma_1])\Phi_{\sigma_2}=([\mathbf{x},\mathbf{y}])\Phi_{\sigma_1}\Phi_{\sigma_2}
\end{equation*}
for any $[\mathbf{x},\mathbf{y}],[\mathbf{u},\mathbf{v}]\in{\mathscr{C}}(p,q)^n$ and any $\sigma,\sigma_1,\sigma_2 \in \mathscr{S}_n$, the following proposition holds:

\begin{proposition}\label{proposition-2.18}
For arbitrary positive integer $n$ the map $\Phi$ is an injective homomorphism from $\mathscr{S}_n$ into the group $\operatorname{\mathbf{Aut}}({{\mathscr{C}^n}(p,q)})$ of automorphisms of the semigroup ${\mathscr{C}}(p,q)^n$.
\end{proposition}

\begin{theorem}\label{theorem-2.19}
For arbitrary positive integer $n$ the semigroup $\mathscr{I\!\!P\!F}(\mathbb{N}^n)$ is isomorphic to the semidirect product $\mathscr{S}_n\ltimes_{\Phi}{\mathscr{C}}(p,q)^n$ of the semigroup ${\mathscr{C}}(p,q)^n$ by the group  $\mathscr{S}_n$.
\end{theorem}

\begin{proof}
We define the map $\Psi\colon \mathscr{I\!\!P\!F}(\mathbb{N}^n)\to\mathscr{S}_n \ltimes_{\Phi}{\mathscr{C}}(p,q)^n$ in the following way:
\begin{equation*}
(\alpha)\Psi=(\sigma_\alpha,[(\mathbf{x})\sigma_\alpha,\mathbf{y}]),
\end{equation*}
for $\alpha\in\mathscr{I\!\!P\!F}(\mathbb{N}^n)$, where ${\uparrow}\mathbf{x}=\operatorname{dom}\alpha$ and ${\uparrow}\mathbf{y}=\operatorname{ran}\alpha$.
Since $\sigma_\alpha$ is a bijection, we have that $\Psi$ is a bijection as well.

For any $\alpha,\beta \in\mathscr{I\!\!P\!F}(\mathbb{N}^n)$ with  $\operatorname{dom}\alpha={\uparrow}\mathbf{x}$, $\operatorname{ran}\alpha={\uparrow}\mathbf{y}$,  $\operatorname{dom}\beta={\uparrow}\mathbf{u}$, $\operatorname{ran}\beta={\uparrow}\mathbf{v}$, by Lemma~\ref{lemma-2.15}, we have that
\begin{equation*}
\begin{split}
(\alpha\beta)\Psi=& \left(\sigma_\alpha\sigma_\beta, [((\max\{\mathbf{y},\mathbf{u}\}-\mathbf{y})\sigma_\alpha^{-1}+\mathbf{x})\sigma_\alpha\sigma_\beta, (\max\{\mathbf{y},\mathbf{u}\}-\mathbf{u})\sigma_\beta+\mathbf{v}]\right)=\\
=&\left(\sigma_\alpha\sigma_\beta,[\max\{(\mathbf{y})\sigma_\beta,(\mathbf{u})\sigma_\beta\}-(\mathbf{y})\sigma_\beta+(\mathbf{x})\sigma_\alpha\sigma_\beta, \max\{(\mathbf{y})\sigma_\beta,(\mathbf{u})\sigma_\beta\}-(\mathbf{u})\sigma_\beta+\mathbf{v}]\right)=\\
=&\left(\sigma_\alpha\sigma_\beta,([(\mathbf{x})\sigma_\alpha,\mathbf{y}])\sigma_\beta*[(\mathbf{u})\sigma_\beta,\mathbf{v}]\right)=\\
=&\left(\sigma_\alpha,[(\mathbf{x})\sigma_\alpha,\mathbf{y}])(\sigma_\beta,[(\mathbf{u})\sigma_\beta,\mathbf{v}]\right)=\\
=&\;(\alpha)\Psi(\beta)\Psi,
\end{split}
\end{equation*}
and hence $\Psi$ is an isomorphism.
\end{proof}

Every inverse semigroup $S$ admits the \emph{least group} congruence $\mathfrak{C}_{\mathbf{mg}}$ (see \cite[Section~III]{Petrich-1984}):
\begin{equation*}
    s\mathfrak{C}_{\mathbf{mg}} t \quad \hbox{if and only if \quad there exists an idempotent} \quad e\in S \quad \hbox{such that} \quad se=te, \qquad s,t\in S.
\end{equation*}

Later, for any $\alpha\in\mathscr{I\!\!P\!F}(\mathbb{N}^n)$ put $\left(\sigma_\alpha,\left[(\mathbf{x}_\alpha)\sigma_\alpha,\mathbf{y}_\alpha\right]\right)= (\alpha)\Psi$ is the image of the element $\alpha$ by the isomorphism $\Psi\colon \mathscr{I\!\!P\!F}(\mathbb{N}^n)\to\mathscr{S}_n \ltimes_{\Phi}{\mathscr{C}}(p,q)^n$ which is defined in the proof of Theorem~\ref{theorem-2.19}.

The following theorem describes the least group congruence on the semigroup $\mathscr{I\!\!P\!F}(\mathbb{N}^n)$.

\begin{theorem}\label{theorem-2.20}
Let $n$ be an arbitrary positive integer. Then $\alpha\mathfrak{C}_{\mathbf{mg}} \beta$ in the semigroup $\mathscr{I\!\!P\!F}(\mathbb{N}^n)$ if and only if
\begin{equation*}
  \sigma_\alpha=\sigma_\beta \qquad \hbox{and} \qquad (\mathbf{x}_\alpha)\sigma_\alpha-\mathbf{y}_\alpha=(\mathbf{x}_\beta)\sigma_\beta-\mathbf{y}_\beta.
\end{equation*}
\end{theorem}

\begin{proof}
First we observe that if $\varepsilon$ is an idempotent in $\mathscr{I\!\!P\!F}(\mathbb{N}^n)$ then Proposition~\ref{proposition-2.16}$(i)$ and the definition of the map $\Psi\colon \mathscr{I\!\!P\!F}(\mathbb{N}^n)\to\mathscr{S}_n \ltimes_\Phi{\mathscr{C}}(p,q)^n$ imply that $\sigma_\varepsilon$ is the identity permutation and $\mathbf{x}_\varepsilon=\mathbf{y}_\varepsilon$.
Then the following calculations
\begin{equation*}
\begin{split}
  \left(\sigma_\alpha,\left[(\mathbf{x}_\alpha)\sigma_\alpha,\mathbf{y}_\alpha\right]\right)& \left(\sigma_\varepsilon,\left[(\mathbf{x}_\varepsilon)\sigma_\varepsilon,\mathbf{x}_\varepsilon\right]\right)= \\  =&\left(\sigma_\alpha\sigma_\varepsilon,\left[\max\{(\mathbf{y}_\alpha)\sigma_\varepsilon,\mathbf{x}_\varepsilon\}-(\mathbf{y}_\alpha)\sigma_\varepsilon +((\mathbf{x}_\alpha)\sigma_\alpha)\sigma_\varepsilon, \max\{(\mathbf{y}_\alpha)\sigma_\varepsilon,\mathbf{x}_\varepsilon\}-\mathbf{x}_\varepsilon+\mathbf{x}_\varepsilon\right]\right)= \\
   = & \left(\sigma_\alpha,\left[\max\{\mathbf{y}_\alpha,\mathbf{x}_\varepsilon\}-\mathbf{y}_\alpha+(\mathbf{x}_\alpha)\sigma_\alpha, \max\{\mathbf{y}_\alpha,\mathbf{x}_\varepsilon\}\right]\right),\\
  \left(\sigma_\beta,\left[(\mathbf{x}_\beta)\sigma_\beta,\mathbf{y}_\beta\right]\right)& \left(\sigma_\varepsilon,\left[(\mathbf{x}_\varepsilon)\sigma_\varepsilon,\mathbf{x}_\varepsilon\right]\right)= \\ =&\left(\sigma_\beta\sigma_\varepsilon,\left[\max\{(\mathbf{y}_\beta)\sigma_\varepsilon,\mathbf{x}_\varepsilon\}-(\mathbf{y}_\beta)\sigma_\varepsilon +((\mathbf{x}_\beta)\sigma_\beta)\sigma_\varepsilon, \max\{(\mathbf{y}_\beta)\sigma_\varepsilon,\mathbf{x}_\varepsilon\}-\mathbf{x}_\varepsilon+\mathbf{x}_\varepsilon\right]\right)= \\
   = & \left(\sigma_\beta,\left[\max\{\mathbf{y}_\beta,\mathbf{x}_\varepsilon\}-\mathbf{y}_\beta+(\mathbf{x}_\beta)\sigma_\beta, \max\{\mathbf{y}_\beta,\mathbf{x}_\varepsilon\}\right]\right),
\end{split}
\end{equation*}
imply that for the  idempotent $\varepsilon\in\mathscr{I\!\!P\!F}(\mathbb{N}^n)$ the
 equality $\alpha\varepsilon=\beta\varepsilon$ holds  if and only if
\begin{equation*}
  \sigma_\alpha=\sigma_\beta \qquad \hbox{and} \qquad (\mathbf{x}_\alpha)\sigma_\alpha-\mathbf{y}_\alpha=(\mathbf{x}_\beta)\sigma_\beta-\mathbf{y}_\beta.
\end{equation*}
This completes the proof of the theorem.
\end{proof}

For any positive integer $n$, an arbitrary permutation $\sigma\colon\left\{1,\ldots,n\right\}\to\left\{1,\ldots,n\right\}$ and an ordered collection $\mathbf{z}=(z_1,\ldots,z_n)$ of integers  we put
\begin{equation*}\label{eq-2.2}
(\mathbf{z})\sigma =\left(z_{(1)\sigma^{-1}},\ldots,z_{(n)\sigma^{-1}}\right).
\end{equation*}
Let $\operatorname{\mathbf{Aut}}(\mathbb{Z}^n) $ be the group of automorphisms of the direct $n$-the power of the additive group of integers $\mathbb{Z}(+)$. Next we define a map $\Theta\colon \mathscr{S}_n \to\operatorname{\mathbf{Aut}}(\mathbb{Z}^n)$ in the following way. We put $(\sigma)\Theta=\Theta_\sigma$ is the map from $\mathbb{Z}^n$ into $\mathbb{Z}^n$ which is defined by the formula
\begin{equation*}
 (\mathbf{z})\Theta_\sigma=(\mathbf{z})\sigma.
\end{equation*}
It is obvious that the maps $\Theta$ and $\Theta_\sigma$ are injective. Since
\begin{equation*}
  (\mathbf{z}+\mathbf{v})\sigma=(\mathbf{z})\sigma+(\mathbf{v})\sigma
\end{equation*}
and
\begin{equation*}
 (\mathbf{z})\Theta_{\sigma_1\sigma_2}=(\mathbf{z})(\sigma_1\sigma_2)=((\mathbf{z})\sigma_1)\sigma_2=((\mathbf{z})\Theta_{\sigma_1})\Theta_{\sigma_2},
\end{equation*}
for any $\mathbf{z}=(z_1,\ldots,z_n)$ and $\mathbf{v}=(v_1,\ldots,v_n)$ from the direct $n$-th power of the group $\mathbb{Z}(+)$ and any $\sigma,\sigma_1,\sigma_2\in \mathscr{S}_n$, we have that so defined map $\Theta\colon \mathscr{S}_n \to\operatorname{\mathbf{Aut}}(\mathbb{Z}^n)$ is an injective homomorphism.

\begin{theorem}\label{theorem-2.21}
For an arbitrary positive integer $n$ the quotient semigroup $\mathscr{I\!\!P\!F}(\mathbb{N}^n)/\mathfrak{C}_{\mathbf{mg}}$ is isomorphic to the semidirect product $\mathscr{S}_n\ltimes_{\Theta}(\mathbb{Z}(+))^n$ of the direct $n$-th power of the additive group of integers $(\mathbb{Z}(+))^n$ by the group of permutation  $\mathscr{S}_n$.
\end{theorem}

\begin{proof}
We define a map $\Upsilon\colon\mathscr{I\!\!P\!F}(\mathbb{N}^n)\to \mathscr{S}_n\ltimes_{\Theta}(\mathbb{Z}(+))^n$ in the following way. If $\left(\sigma_\alpha,\left[(\mathbf{x}_\alpha)\sigma_\alpha,\mathbf{y}_\alpha\right]\right)=(\alpha)\Psi$ is the image of $\alpha\in\mathscr{I\!\!P\!F}(\mathbb{N}^n)$ under the isomorphism $\Psi\colon \mathscr{I\!\!P\!F}(\mathbb{N}^n)\to\mathscr{S}_n \ltimes_{\Phi}{\mathscr{C}}(p,q)^n$ which is defined in the proof of Theorem~\ref{theorem-2.19}, then we put $(\alpha)\Upsilon=\left(\sigma_\alpha,(\mathbf{x}_\alpha)\sigma_\alpha-\mathbf{y}_\alpha\right)$.

For any $\alpha,\beta \in\mathscr{I\!\!P\!F}(\mathbb{N}^n)$ with  $\operatorname{dom}\alpha={\uparrow}\mathbf{x}_\alpha$, $\operatorname{ran}\alpha={\uparrow}\mathbf{y}_\alpha$,  $\operatorname{dom}\beta={\uparrow}\mathbf{x}_\beta$, $\operatorname{ran}\beta={\uparrow}\mathbf{y}_\beta$,  by Lemma~\ref{lemma-2.15}, we have that
\begin{equation*}
\begin{split}
(\alpha\beta)\Upsilon=& \left(\sigma_\alpha\sigma_\beta, ((\max\{\mathbf{y}_\alpha,\mathbf{x}_\beta\}-\mathbf{y}_\alpha)\sigma_\alpha^{-1}+\mathbf{x}_\alpha)\sigma_\alpha\sigma_\beta- (\max\{\mathbf{y}_\alpha,\mathbf{x}_\beta\}-\mathbf{x}_\beta)\sigma_\beta-\mathbf{y}_\beta\right)=\\
=&\left(\sigma_\alpha\sigma_\beta,\max\{(\mathbf{y}_\alpha)\sigma_\beta,(\mathbf{x}_\beta)\sigma_\beta\}-(\mathbf{y}_\alpha)\sigma_\beta+ (\mathbf{x}_\alpha)\sigma_\alpha\sigma_\beta- \max\{(\mathbf{y}_\alpha)\sigma_\beta,(\mathbf{x}_\beta)\sigma_\beta\} +(\mathbf{x}_\beta)\sigma_\beta-\mathbf{y}_\beta\right)=\\
=&\left(\sigma_\alpha\sigma_\beta,(\mathbf{x}_\alpha)\sigma_\alpha\sigma_\beta-(\mathbf{y}_\alpha)\sigma_\beta+ (\mathbf{x}_\beta)\sigma_\beta-\mathbf{y}_\beta\right)=\\
=&\left(\sigma_\alpha\sigma_\beta,((\mathbf{x}_\alpha)\sigma_\alpha-\mathbf{y}_\alpha)\sigma_\beta+ ((\mathbf{x}_\beta)\sigma_\beta-\mathbf{y}_\beta)\right)=\\
=&\left(\sigma_\alpha,(\mathbf{x}_\alpha)\sigma_\alpha-\mathbf{y}_\alpha\right)\cdot\left(\sigma_\beta,(\mathbf{x}_\beta)\sigma_\beta-\mathbf{y}_\beta\right)=\\
=&\;(\alpha)\Upsilon\cdot(\beta)\Upsilon,
\end{split}
\end{equation*}
and hence $\Upsilon$ is a  homomorphism. It is obvious that the map $\Upsilon\colon\mathscr{I\!\!P\!F}(\mathbb{N}^n)\to \mathscr{S}_n\ltimes_{\Theta}(\mathbb{Z}(+))^n$ is surjective. Also, Theorem~\ref{theorem-2.20} implies that $\alpha\mathfrak{C}_{\mathbf{mg}} \beta$ in  $\mathscr{I\!\!P\!F}(\mathbb{N}^n)$ if and only if $(\alpha)\Upsilon=(\beta)\Upsilon$. This implies that the homomorphism $\Upsilon$ generates the congruences $\mathfrak{C}_{\mathbf{mg}}$ on $\mathscr{I\!\!P\!F}(\mathbb{N}^n)$.
\end{proof}

Every inverse semigroup $S$ admits a partial order:
\begin{equation*}
  a\preccurlyeq b \qquad \hbox{if and only if there exists} \qquad e\in E(S) \quad \hbox{such that} \quad a=b e, \qquad a,b \in S.
\end{equation*}
So defined order is called \emph{the natural partial order} on $S$. We observe that $a\preccurlyeq b$ in an inverse semigroup $S$ if and only if $a=f b$ for some $f\in E(S)$ (see \cite[Lemma~1.4.6]{Lawson-1998}).

If $\varepsilon$ is an idempotent in $\mathscr{I\!\!P\!F}(\mathbb{N}^n)$ then Proposition~\ref{proposition-2.16}$(i)$ and the definition of the isomorphism $\Psi\colon \mathscr{I\!\!P\!F}(\mathbb{N}^n)\to\mathscr{S}_n \ltimes_\Phi{\mathscr{C}}(p,q)^n$ imply that $\sigma_\varepsilon$ is the identity permutation and $\mathbf{x}_\varepsilon=\mathbf{y}_\varepsilon$.
Then for any $\alpha\in\mathscr{I\!\!P\!F}(\mathbb{N}^n)$ with $\left(\sigma_\alpha,\left[(\mathbf{x}_\alpha)\sigma_\alpha,\mathbf{y}_\alpha\right]\right)=(\alpha)\Psi$ we have that
\begin{equation}\label{eq-2.4}
\begin{split}
  \left(\sigma_\alpha{,}\left[(\mathbf{x}_\alpha)\sigma_\alpha,\mathbf{y}_\alpha\right]\right)& \left(\sigma_\varepsilon{,}\left[\mathbf{x}_\varepsilon{,}\mathbf{x}_\varepsilon\right]\right)= \\
=& \left(\sigma_\alpha\sigma_\varepsilon,\left[\max\{(\mathbf{y}_\alpha)\sigma_\varepsilon,\mathbf{x}_\varepsilon\}-(\mathbf{y}_\alpha)\sigma_\varepsilon +((\mathbf{x}_\alpha)\sigma_\alpha)\sigma_\varepsilon, \max\{(\mathbf{y}_\alpha)\sigma_\varepsilon,\mathbf{x}_\varepsilon\}-\mathbf{x}_\varepsilon+\mathbf{x}_\varepsilon\right]\right)= \\
   = & \left(\sigma_\alpha,\left[\max\{\mathbf{y}_\alpha,\mathbf{x}_\varepsilon\}-\mathbf{y}_\alpha+(\mathbf{x}_\alpha)\sigma_\alpha, \max\{\mathbf{y}_\alpha,\mathbf{x}_\varepsilon\}\right]\right),
\end{split}
\end{equation}

This implies the following proposition, which describes the natural partial order on the semigroup $\mathscr{I\!\!P\!F}(\mathbb{N}^n)$.

\begin{proposition}\label{proposition-2.22}
Let $n$ be an arbitrary positive integer and $\alpha,\beta\in\mathscr{I\!\!P\!F}(\mathbb{N}^n)$. Then the following conditions are equivalent:
\begin{itemize}
  \item[$(i)$] $\alpha\preccurlyeq\beta$;
  \item[$(ii)$] $\sigma_\alpha=\sigma_\beta$, $(\mathbf{x}_\alpha)\sigma_\alpha-\mathbf{y}_\alpha=(\mathbf{x}_\beta)\sigma_\beta-\mathbf{y}_\beta$ and $\mathbf{x}_\alpha\leqslant\mathbf{x}_\beta$ in the poset $\left(\mathbb{N}^n,\leqslant\right)$;
  \item[$(iii)$] $\sigma_\alpha=\sigma_\beta$, $(\mathbf{x}_\alpha)\sigma_\alpha-\mathbf{y}_\alpha=(\mathbf{x}_\beta)\sigma_\beta-\mathbf{y}_\beta$ and $\mathbf{y}_\alpha\leqslant\mathbf{y}_\beta$ in the poset $\left(\mathbb{N}^n,\leqslant\right)$.
\end{itemize}
\end{proposition}

An inverse semigroup $S$ is said to be \emph{$E$-unitary} if $a e\in E(S)$ for some $e\in E(S)$ implies that $a\in E(S)$ \cite{Lawson-1998}. $E$-unitary inverse semigroups were introduced by Siat\^{o} in \cite{Saito-1965}, where they were called ``\emph{proper ordered inverse semigroups}''.

Formula \eqref{eq-2.4} implies that if the element $\left(\sigma_\alpha{,}\left[(\mathbf{x}_\alpha)\sigma_\alpha,\mathbf{y}_\alpha\right]\right)\cdot \left(\sigma_\varepsilon,\left[\mathbf{x}_\varepsilon{,}\mathbf{x}_\varepsilon\right]\right)$ is an idempotent in the semidirect product $\mathscr{S}_n \ltimes_\Phi{\mathscr{C}}(p,q)^n$ then so is $\left(\sigma_\alpha{,}\left[(\mathbf{x}_\alpha)\sigma_\alpha,\mathbf{y}_\alpha\right]\right)$. This implies the following

\begin{corollary}\label{proposition-2.23}
For an arbitrary positive integer $n$ the inverse semigroup $\mathscr{I\!\!P\!F}(\mathbb{N}^n)$ is $E$-unitary.
\end{corollary}

An inverse semigroup $S$ is called \emph{$F$-inverse}, if the $\mathfrak{C}_{\textsf{mg}}$-class $s_{\mathfrak{C}_{\textsf{mg}}}$ of each element $s$ has the top (biggest) element with the respect to the natural partial order on $S$ \cite{McFadden-Carroll-1971}.

\begin{proposition}\label{proposition-2.24}
For an arbitrary positive integer $n$ the semigroup $\mathscr{I\!\!P\!F}(\mathbb{N}^n)$ is an $F$-inverse semigroup.
\end{proposition}

\begin{proof}
Fix an arbitrary element $\beta_0=(\sigma,[\mathbf{x},\mathbf{y}])\in \mathscr{S}_n \ltimes_\Phi{\mathscr{C}}(p,q)^n$, where $\mathbf{x}=(x_1,\ldots,x_n)$ and $\mathbf{y}=(y_1,\ldots,y_n)$. For every positive integer $k$ we put $\beta_k=(\sigma,[\mathbf{x}-\mathbf{k},\mathbf{y}-\mathbf{k}])$, where $\mathbf{x}-\mathbf{k}=(x_1-k,\ldots,x_n-k)$ and $\mathbf{y}-\mathbf{k}=(y_1-k,\ldots,y_n-k)$. It is obvious that there exists a (biggest) positive integer $k_0$ such that
\begin{equation*}
  (x_1-k_0,\ldots,x_n-k_0)\in\mathbb{N}^n \qquad \hbox{and} \qquad (y_1-k_0,\ldots,y_n-k_0)\in\mathbb{N}^n.
\end{equation*}
Then Theorem~\ref{theorem-2.20} and Proposition~\ref{proposition-2.22} imply that the element $\beta_{k_0}$ is the biggest element in the $\mathfrak{C}_{\mathbf{mg}}$-class of the element $\beta_0$ in $\mathscr{S}_n \ltimes_\Phi{\mathscr{C}}(p,q)^n$.
\end{proof}

\begin{proposition}\label{proposition-2.25}
For every $\alpha,\beta\in\mathscr{I\!\!P\!F}(\mathbb{N}^n)$, both sets
 $$
\{\chi\in\mathscr{I\!\!P\!F}(\mathbb{N}^n)\colon \alpha\cdot\chi=\beta\}
\qquad \hbox{and} \qquad
\{\chi\in\mathscr{I\!\!P\!F}(\mathbb{N}^n) \colon
\chi\cdot\alpha=\beta\}
 $$
are finite. Consequently, every right translation and every left translation by an element of the semigroup $\mathscr{I\!\!P\!F}(\mathbb{N}^n)$ is a finite-to-one map.
\end{proposition}

\begin{proof}
We  show that the set $A=\{\chi\in\mathscr{I\!\!P\!F}(\mathbb{N}^n) \colon
\chi\cdot\alpha=\beta\}$ is finite. The proof of the statement that the set $\{\chi\in\mathscr{I\!\!P\!F}(\mathbb{N}^n)\colon \alpha\cdot\chi=\beta\}$ is finite, is similar.

It is obvious that $A$ is a subset of the set $B=\{\chi\in\mathscr{I\!\!P\!F}(\mathbb{N}^n) \colon \chi\cdot\alpha\alpha^{-1}=\beta\alpha^{-1}\}$. Then $B$ is a subset of $C=\{\xi\in\mathscr{I\!\!P\!F}(\mathbb{N}^n) \colon \xi\preccurlyeq\beta\alpha^{-1}\}$. Since every principal ideal in the poset $\left(\mathbb{N}^n,\leqslant\right)$ is finite, Proposition~\ref{proposition-2.25} implies that $C$ is finite, and hence so is $A$.
\end{proof}

\section{On a semitopological semigroup $\mathscr{I\!\!P\!F}(\mathbb{N}^n)$}

The following theorem generalizes the Bertman-West result from \cite{Bertman-West-1976} (and hence Eberhart-Selden result from~\cite{Eberhart-Selden-1969}) for the semigroup $\mathscr{I\!\!P\!F}(\mathbb{N}^n)$.

\begin{theorem}\label{theorem-3.1}
For any positive integer $n$ every Hausdorff shift-continuous topology on the semigroup $\mathscr{I\!\!P\!F}(\mathbb{N}^n)$ is discrete.
\end{theorem}

\begin{proof}
Fix an arbitrary shift continuous Hausdorff topology $\tau$ on $\mathscr{I\!\!P\!F}(\mathbb{N}^n)$. Since for every idempotent $\varepsilon\in \mathscr{I\!\!P\!F}(\mathbb{N}^n)$ the left and right shifts $\mathfrak{l}_\varepsilon\colon \mathscr{I\!\!P\!F}(\mathbb{N}^n) \to \mathscr{I\!\!P\!F}(\mathbb{N}^n)\colon x\mapsto \varepsilon\cdot x$ and $\mathfrak{r}_\varepsilon\colon \mathscr{I\!\!P\!F}(\mathbb{N}^n) \to \mathscr{I\!\!P\!F}(\mathbb{N}^n)\colon x\mapsto x\cdot\varepsilon$ are continuous maps, the Hausdorffness of $\left(\mathscr{I\!\!P\!F}(\mathbb{N}^n),\tau\right)$ and \cite[1.5.c]{Engelking-1989} implies that the principal ideals $\varepsilon\mathscr{I\!\!P\!F}(\mathbb{N}^n)$ and $\mathscr{I\!\!P\!F}(\mathbb{N}^n)\varepsilon$ are closed subsets in $\left(\mathscr{I\!\!P\!F}(\mathbb{N}^n),\tau\right)$.

For any positive integer $i=1,\ldots,n$ put $\varepsilon_i$ is the identity map of the subset ${\uparrow}(1,\ldots,\underbrace{2}_{i\hbox{\footnotesize{-th}}},\ldots, 1)$ of the poset $\left(\mathbb{N}^n,\leqslant\right)$. It is clear that
\begin{equation*}
  H(\mathbb{I})=\mathscr{I\!\!P\!F}(\mathbb{N}^n)\setminus \left(\varepsilon_1\mathscr{I\!\!P\!F}(\mathbb{N})^n\cup\cdots \varepsilon_n\mathscr{I\!\!P\!F}(\mathbb{N}^n)\cup\mathscr{I\!\!P\!F}(\mathbb{N}^n)\varepsilon_1\cup\cdots\cup\mathscr{I\!\!P\!F}(\mathbb{N}^n)\varepsilon_n \right).
\end{equation*}
The above part of the proof implies that $H(\mathbb{I})$ is an open subset of $\left(\mathscr{I\!\!P\!F}(\mathbb{N}^n),\tau\right)$, and by Theorem~\ref{theorem-2.5}, $H(\mathbb{I})$ is a finite discrete open subspace of $\left(\mathscr{I\!\!P\!F}(\mathbb{N}^n),\tau\right)$.

Since  the semigroup $\mathscr{I\!\!P\!F}(\mathbb{N}^n)$ is simple (see Proposition~\ref{proposition-2.1}$(vii)$), for an arbitrary $\alpha\in\mathscr{I\!\!P\!F}(\mathbb{N}^n)$ there exist $\beta,\gamma\in \mathscr{I\!\!P\!F}(\mathbb{N}^n)$ such that $\beta\alpha\gamma=\mathbb{I}$. Then Proposition~\ref{proposition-2.25} implies that the point $\alpha$ has a finite open neighbourhood in $\left(\mathscr{I\!\!P\!F}(\mathbb{N}^n),\tau\right)$ and hence $\alpha$ is an isolated point in $\left(\mathscr{I\!\!P\!F}(\mathbb{N}^n),\tau\right)$.
\end{proof}

The following theorem generalizes Theorem~I.3 from \cite{Eberhart-Selden-1969}.

\begin{theorem}\label{theorem-3.2}
If for some positive integer $n$ the semigroup $\mathscr{I\!\!P\!F}(\mathbb{N}^n)$ is dense in a Hausdorff semitopological semigroup $(S,\cdot)$ and $I=S\setminus\mathscr{I\!\!P\!F}(\mathbb{N}^n)\neq\varnothing$, then $I$ is a two-sided ideal in $S$.
\end{theorem}

\begin{proof}
Fix an arbitrary element $y\in I$. If $x\cdot y=z\notin I$ for some $x\in\mathscr{I\!\!P\!F}(\mathbb{N}^n)$ then there exists an open neighbourhood $U(y)$ of the point $y$ in the space $S$ such that $\{x\}\cdot U(y)=\{z\}\subset\mathscr{I\!\!P\!F}(\mathbb{N}^n)$. By Proposition~\ref{proposition-2.25} the open neighbourhood $U(y)$ should contain finitely many elements of the semigroup $\mathscr{I\!\!P\!F}(\mathbb{N}^n)$ which contradicts our assumption. Hence $x\cdot y\in I$ for all $x\in \mathscr{I\!\!P\!F}(\mathbb{N}^n)$ and $y\in I$. The proof of the statement that $y\cdot x\in I$ for all $x\in\mathscr{I\!\!P\!F}(\mathbb{N}^n)$ and $y\in I$ is similar.

Suppose to the contrary that $x\cdot y=w\notin I$ for some $x,y\in I$. Then $w\in \mathscr{I\!\!P\!F}(\mathbb{N}^n)$ and the separate continuity of the semigroup operation in $S$ yields open neighbourhoods $U(x)$ and $U(y)$ of the points $x$ and $y$, respectively, such that $\{x\}\cdot U(y)=\{w\}$ and $U(x)\cdot \{y\}=\{w\}$. Since both neighbourhoods $U(x)$ and $U(y)$ contain infinitely many elements of the semigroup $\mathscr{I\!\!P\!F}(\mathbb{N}^n)$,  equalities $\{x\}\cdot U(y)=\{w\}$ and $U(x)\cdot \{y\}=\{w\}$ don't hold, because $\{x\}\cdot \left(U(y)\cap\mathscr{I\!\!P\!F}(\mathbb{N}^n)\right)\subseteq I$. The obtained contradiction implies that $x\cdot y\in I$.
\end{proof}

We recall that a topological space X is said to be:
\begin{itemize}
    \item \emph{compact} if every open cover of $X$ contains a finite subcover;
    \item \emph{countably compact} if each closed discrete subspace of $X$ is finite;
    \item \emph{feebly compact} if each locally finite open cover of $X$ is finite;
    \item \emph{pseudocompact} if $X$ is Tychonoff and each continuous real-valued function on $X$ is bounded.
\end{itemize}
According to Theorem~3.10.22 of \cite{Engelking-1989}, a Tychonoff topological space $X$ is feebly compact if and only if $X$ is pseudocompact. Also, a Hausdorff topological space $X$ is feebly compact if and only if every locally finite family of non-empty open subsets of $X$ is finite. Every compact space is countably compact and every countably compact space is feebly compact (see \cite{Arkhangelskii-1992}).

A topological semigroup $S$ is called
\emph{$\Gamma$-compact} if for every $x\in S$ the closure of the set
$\{x,x^2,x^3,\ldots\}$ is compact in $S$ (see
\cite{Hildebrant-Koch-1986}). Since  for every positive integer $n$ the semigroup $\mathscr{I\!\!P\!F}(\mathbb{N}^n)$ contains the bicyclic semigroup as a subsemigroup (see Proposition~\ref{proposition-2.7}), the results obtained in
\cite{Anderson-Hunter-Koch-1965}, \cite{Banakh-Dimitrova-Gutik-2009},
\cite{Banakh-Dimitrova-Gutik-2010}, \cite{Gutik-Repovs-2007},
\cite{Hildebrant-Koch-1986},  Theorems~\ref{theorem-2.12} and \ref{theorem-2.21}  imply the following

\begin{corollary}\label{corollary-3.4}
Let $n$ be an arbitrary non-negative integer. If a Hausdorff topological semigroup $S$ satisfies one of the following conditions:
\begin{itemize}
  \item[$(i)$] $S$ is compact;
  \item[$(ii)$] $S$ is $\Gamma$-compact;
  \item[$(iii)$] $S$ is a countably compact topological inverse semigroup;
  \item[$(iv)$] the square $S\times S$ is countably compact;
  \item[$(v)$] the square $S\times S$ is a Tychonoff pseudocompact space,
\end{itemize}
then $S$ does not contain the semigroup $\mathscr{I\!\!P\!F}(\mathbb{N}^n)$ and for every homomorphism $\mathfrak{h}\colon\mathscr{I\!\!P\!F}(\mathbb{N}^n)\to S$ the image $(\mathscr{I\!\!P\!F}(\mathbb{N}^n))\mathfrak{h}$ is a subgroup of $S$. Moreover, for every homomorphism $\mathfrak{h}\colon\mathscr{I\!\!P\!F}(\mathbb{N}^n)\to S$ there exists a unique homomorphism $\mathfrak{u}_\mathfrak{h}\colon \mathscr{S}_n\ltimes_{\Theta}(\mathbb{Z}(+))^n\to S$ such that the following diagram
\begin{equation*}
\xymatrix{
\mathscr{I\!\!P\!F}(\mathbb{N}^n)\ar[rr]^{\mathfrak{h}}\ar[dd]_\Upsilon && S\\
&&\\
\mathscr{S}_n\ltimes_{\Theta}(\mathbb{Z}(+))^n\ar[rruu]_{\mathfrak{u}_{\mathfrak{h}}}
}
\end{equation*}
commutes.
\end{corollary}

\begin{proposition}\label{proposition-3.5}
Let $n$ be an arbitrary positive integer. Let $S$ be a Hausdorff topological semigroup which contains $\mathscr{I\!\!P\!F}(\mathbb{N}^n)$ as a dense subsemigroup. Then for every $c\in\mathscr{I\!\!P\!F}(\mathbb{N}^n)$ the set
\begin{equation*}
    D_c=\left\{(x,y)\in\mathscr{I\!\!P\!F}(\mathbb{N}^n)\times\mathscr{I\!\!P\!F}(\mathbb{N}^n)\colon x\cdot y=c\right\}
\end{equation*}
is an open-and-closed subset of $S\times S$.
\end{proposition}

\begin{proof}
By Theorem~\ref{theorem-3.1}, $\mathscr{I\!\!P\!F}(\mathbb{N}^n)$ is a discrete subspace of $S$ and hence Theorem~3.3.9 of \cite{Engelking-1989} implies that $\mathscr{I\!\!P\!F}(\mathbb{N}^n)$ is an open subspace of $S$. Then the continuity of the semigroup operation of $S$ implies that $D_c$ is an open subset of $S\times S$ for every $c\in\mathscr{I\!\!P\!F}(\mathbb{N}^n)$.

Suppose that there exists $c\in\mathscr{I\!\!P\!F}(\mathbb{N}^n)$ such that $D_c$ is a non-closed subset of $S\times S$. Then there exists an accumulation point $(a,b)\in S\times S$ of the set $D_c$. The continuity of the semigroup operation in $S$ implies that $a\cdot b=c$. But $\mathscr{I\!\!P\!F}(\mathbb{N}^n)\times \mathscr{I\!\!P\!F}(\mathbb{N}^n)$ is a discrete subspace of $S\times S$ and hence by Theorem~\ref{theorem-3.2} the points $a$ and $b$ belong to the two-sided
ideal $I=S\setminus \mathscr{I\!\!P\!F}(\mathbb{N}^n)$. This implies that the product $a\cdot b\in S\setminus \mathscr{I\!\!P\!F}(\mathbb{N}^n)$ cannot be equal to the element $c$.
\end{proof}

\begin{theorem}\label{theorem-3.6}
Let $n$ be an arbitrary positive integer. If a Hausdorff topological semigroup $S$ contains $\mathscr{I\!\!P\!F}(\mathbb{N}^n)$ as a dense subsemigroup then the square $S\times S$ is not feebly compact.
\end{theorem}

\begin{proof}
By Proposition~\ref{proposition-3.5}, for every $c\in\mathscr{I\!\!P\!F}(\mathbb{N}^n)$ the square $S\times S$ contains an open-and-closed discrete subspace $D_c$. In the case when $c$ is the unit $\mathbb{I}$ of $\mathscr{I\!\!P\!F}(\mathbb{N}^n)$, by Corollary~\ref{corollary-2.8}, there exists a subsemigroup ${\mathscr{C}}$ of $\mathscr{I\!\!P\!F}(\mathbb{N}^n)$ which is isomorphic to the bicyclic monoid ${\mathscr{C}}(p,q)$ and contains $\mathbb{I}$. If we identify the elements of ${\mathscr{C}}$  with the elements the bicyclic monoid ${\mathscr{C}}(p,q)$ by an isomorphism $\mathfrak{h}\colon {\mathscr{C}}(p,q)\to \mathscr{C}$ then the subspace $D_c$ contains an infinite subset $\left\{\left((q^i)\mathfrak{h},(p^i)\mathfrak{h}\right)\colon i\in\mathbb{N}_0\right\}$ and hence the set $D_c$ is infinite. This implies that the square $S\times S$ is not feebly compact.
\end{proof}

\section*{Acknowledgements}

The authors are grateful to the referee and the editor for several useful comments and suggestions.

\end{document}